\documentclass[reqno,11pt]{amsart}
\usepackage[margin=1in]{geometry}

\newif\iffull
\fullfalse

\pdfoutput=1
  \pdfpageattr{/Group <</S /Transparency /I true /CS /DeviceRGB>>}

\makeatletter
\def\ps@pprintTitle{%
 \let\@oddhead\@empty
 \let\@evenhead\@empty
 \def\@oddfoot{}%
 \let\@evenfoot\@oddfoot}
\makeatother

\usepackage{hyperref}
\usepackage{amsfonts,amsmath,amssymb,amsthm}
\usepackage{graphicx}
\usepackage{tikz}
\usepackage{mytikz}
\usepackage{caption}
\usepackage[labelfont=rm]{subcaption}
\usepackage{xspace}
\usepackage{pgfplots}
\definecolor{hrcitecolor}{rgb}{0.78,0,0.27}
\definecolor{hrlinkcolor}{rgb}{0.2,0,0.75}
\definecolor{hrurlcolor}{rgb}{0.1,0.1,0.1}
\hypersetup{
   colorlinks   = true,
   citecolor    = hrcitecolor,
   linkcolor    = hrlinkcolor,
   urlcolor = hrurlcolor
}

\usepackage{thmtools, thm-restate}
\usepackage[disable]{todonotes}

\usepackage{dsfont}

\newtheorem{theorem}{Theorem}[section]
\newtheorem{lemma}[theorem]{Lemma}
\newtheorem{proposition}[theorem]{Proposition}
\newtheorem{corollary}[theorem]{Corollary} 
\theoremstyle{definition}
\newtheorem{definition}[theorem]{Definition}

\theoremstyle{remark}
\newtheorem{remark}[theorem]{Remark}

\newtheorem{example}[theorem]{Example}

\usepackage{algorithmicx}
\usepackage{algorithm}
\usepackage{algpseudocode}

\usepackage{enumerate}

\algdef{SE}[DOWHILE]{Do}{DoWhile}{\algorithmicdo}[1]{\algorithmicwhile\ #1}
\algrenewcommand{\algorithmiccomment}[1]{\hfill $\rhd$ \emph{#1}}%{\hskip1em $\rhd$ \emph{#1}}
\algrenewcommand{\algorithmicrequire}{\textbf{Input:}}
\algrenewcommand{\algorithmicensure}{\textbf{Output:}}
\algnewcommand{\Or}{\textbf{or}}
\algnewcommand{\And}{\textbf{and}}
\algnewcommand{\Not}{\textbf{not}\,}

\newcommand\pvec[2]{\begin{pmatrix} #1 \\ #2 \end{pmatrix}}

\newcommand\TTmax{\mathbb{T}_{\max}}
\newcommand\SetOf[2]{\left\{\left.#1\vphantom{#2}\ \right|\ #2\vphantom{#1}\right\}}
\newcommand\tcone[1]{\operatorname{tcone}(#1)}

\DeclareMathOperator{\supp}{supp}

\newcommand\KK{{\mathbb K}}
\DeclareMathOperator\val{val}
\newcommand\TT{{\mathbb T}}
\newcommand\TTpm{\mathbb T_{\pm}}

\newcommand\RR{{\mathbb R}}
\newcommand\ZZ{{\mathbb Z}}
\newcommand\BB{{\mathbb B}}
\newcommand\trans[1]{{#1}^{\top}}

\newcommand\teq{\mathrel{\vDash}}%\models}}%\vdash}}
\newcommand\seq{\mathrel{\rotatebox[origin=c]{180}{\ensuremath{\vDash}}}} % Dashv}}%\dashv}}
\newcommand\beq{\bowtie} %{\circeq} %\teq\seq}

\newcommand{\TSS}{\mathbb{S}}

\newcommand{\NN}{\mathbb{N}}
\newcommand{\STH}{\mathbb{H}}

\newcommand\Id{\boldsymbol{I}}

\DeclareMathOperator\cancsum{\overline{\oplus}}

\DeclareMathOperator{\Hyp}{Hyp}

\DeclareMathOperator{\sval}{sval}
\DeclareMathOperator{\conv}{conv}
\DeclareMathOperator{\slog}{slog}

\DeclareMathOperator{\sgn}{sgn}
\DeclareMathOperator{\tsgn}{tsgn}
\DeclareMathOperator{\argmax}{argmax}
\DeclareMathOperator{\argmin}{argmin}
\DeclareMathOperator{\tconv}{tconv}
\newcommand\nnker{\text{ker}_+}

\newcommand\sep{\text{sep}_+}
\newcommand\Tleq{\TT_{\leq \mathds{O}}}
\newcommand\Tgeq{\TT_{\geq \mathds{O}}}
\newcommand\Tlt{\TT_{< \mathds{O}}}
\newcommand\Tgt{\TT_{> \mathds{O}}}
\newcommand\Tzero{\TT_{\bullet}}
\newcommand\Zero{\mathds{O}}
\newcommand\Tclosed{\overline{\TT}}
\newcommand\STT{\mathbb{ST}}

\newcommand\Uncomp{\mathcal{U}}
\newcommand\HC{\mathcal{H}}
\newcommand\Hclosed{\overline{\mathcal{H}}}

\DeclareMathOperator{\Co}{Co}

\newcommand\bA{\mathbf{A}}
\newcommand\blambda{\mathbf{\lambda}}

\newcommand\pseries[2]{#1\{\!\!\{#2\}\!\!\}} % coefficients, variable
\makeatletter
\def\input@path{{./Images/}{./}}
\makeatother

\begin{document}

\title{Signed tropical convexity}

%\title{A new tropical duality}
% \dedicatory{}

\begin{abstract}
  We establish a new notion of tropical convexity for signed tropical numbers.
 We provide several equivalent descriptions involving balance relations and
 intersections of open halfspaces as well as the image of a union of polytopes over Puiseux series and hyperoperations.
  Along the way, we deduce a new Farkas lemma and Fourier-Motzkin elimination without the non-negativity restriction on the variables.
  This leads to a Minkowski-Weyl theorem for polytopes over the signed tropical numbers.
\end{abstract}

%\iffull
\author{Georg Loho}
\address{London School of Economics and Political Science. }
\email{g.loho@lse.ac.uk}

\author{L{\'a}szl{\'o} A. V{\'e}gh}
\address{London School of Economics and Political Science. }
\email{l.vegh@lse.ac.uk}

\thanks{Supported by the ERC Starting Grant ScaleOpt. }
%\fi
% MSC2010 classes: 52A30 Variants of convex sets (star-shaped, ($m, n$)-convex, etc.), 14T05 Tropical geometry, 90C05 Linear programming

\maketitle

\section{Introduction}

Tropical convexity is an important notion with applications in several branches of mathematics.
It arises from the usual definition of convexity by replacing $+$ with $\max$ and $\cdot$ with $+$.
This notion has been studied for several years involving different approaches from extremal algebra~\cite{Zimmermann:1977}, idempotent semirings~\cite{CohenGaubertQuadrat:2005}, max-algebra~\cite{ButkovicSchneiderSergeev:2007}, convex analysis~\cite{BriecHorvath:2004}, discrete geometry~\cite{DevelinSturmfels:2004}, matroid theory~\cite{FinkRincon:2015}.
So far, it was mainly studied in $\TTmax = \RR \cup \{-\infty\}$.
Indeed, this is essentially a restriction to the tropical non-negative orthant, as $r \geq -\infty$ for all $r \in \TTmax$, where $-\infty$ is the tropical zero element.
We remedy this restriction by introducing a notion of tropical convexity involving all orthants.
We give our main points of motivation for our generalization. 

Mean payoff games are equivalent to feasibility of a tropical linear inequalities % of the form
\begin{equation} \label{eq:tropical+linear+inequality+system}
  %\bigoplus_{a_{ij} > \Zero} a_{ij} + x_j \geq \bigoplus_{a_{ij} < \Zero} \ominus a_{ij} + x_j
  \bigoplus_{j \in J_i} a_{ij} + x_j \geq \bigoplus_{j \in [n]
    \setminus J_i} a_{ij} + x_j\quad \forall j\in [m]\ ,
\end{equation}
where $a_i,x\in \TTmax^n$, for $i\in [n]$, see~\cite{AkianGaubertGuterman:2012}.
This problem is in NP~$\cap$~co-NP, but no polynomial-time algorithm is known~\cite{GKK:1988}.
Furthermore, the latter feasibility problem is intimately related to
the feasibility problem for classical linear inequality
systems~\cite{Schewe2009,AllamigeonBenchimolGaubertJoswig:2015}.
The tropical linear feasibility problem is also a special scheduling problem~\cite{MoehringSkutellaStork:2004} and it can be considered as a particular disjunctive programming problem~\cite{Balas:1979}. 

In the context of mean payoff games, signs can be used to represent
the two sides of the inequality
\eqref{eq:tropical+linear+inequality+system}. This system can be
concisely written as $A\odot x \geq \Zero, x\geq \Zero$, where $\Zero$
represents the vector of $-\infty$'s, and $A\in \TTpm^{m\times n}$
is a matrix with $A_{ij}=a_{ij}$ if $j\in J_i$ and $A_{ij}=\ominus
a_{ij}$ if $j\notin J_i$. Here, $\ominus a_{ij}$ represents a
tropically negative number (note that $5$ and $-5$ are tropically positive, while
$\ominus 5$ and $\ominus -5$ are tropically negative).

Our main motivation is to extend some of the basic concepts and tools
from convex geometry and polyhedral combinatorics to signed tropical
numbers. Significant work has been done already in this direction,
however, some fundamentals are still missing: in particular, a
satisfactory notion of convexity for signed tropical numbers has not
been given prior to this work.

Such an investigation is motivated both by algorithmic questions
as well as in the context of  recent developments in tropical
geometry. Solving mean payoff games in polynomial time is a major open
question. Developing tropical versions of polyhedral combinatorics
tools may lead to
new approaches to tackle this problem. For example, there is a relatively recent
class of polynomial-time algorithms for linear programming are
naturally formulated as deciding if the origin is in the convex hull of a
set of points, see, e.g.,~\cite{Chubanov:2015}. 
Our convexity notion provides an analogous formulation for the
tropical linear feasibility problem in terms of the signed convex hull
of the coefficient vectors.
%(see Corollary~\ref{cor:zero+in+convex+hull} and Theorem~\ref{thm:farkas+lemma}).

Signed tropical numbers have been used in important recent results in
tropical geometry, for example, in the 
%In the context of tropical geo
%Additionally, this notion is a natural next step following recent developments in tropical geometry.
%The concept of signed tropical numbers, a way to model inverse
%elements for the $\max$-operation, was 
%effectively used in the 
tropicalization of the simplex method~\cite{AllamigeonBenchimolGaubertJoswig:2015}.
The study of real tropicalization of semialgebraic sets~\cite{JellScheidererYu:2018} follows a similar spirit.
Another approach to extend from $\TTmax^d$ involving signs is to deduce the structure of a variety by `unfolding' it from the positive orthant into the other orthants, which is formalized by the patchworking introduced in~\cite{Viro:2006} that has several applications in algebraic geometry.

Furthermore, separation theorems like Farkas' lemma for linear
programming have their easiest formulation in terms of separation from
the origin leading to powerful generalizations to oriented matroids, see~\cite{BachemKern:1992}.
Our approach allows to formulate an analogous theory for tropical linear programming. 
This gives new possibilities for studying tropical normal fans and tropical hyperplane arrangements.
\smallskip

Arithmetic operations can be naturally extended from $\TTmax$ to the
set of signed tropical numbers $\TTpm$. There is however a critical
case of degeneracy, namely, adding a positive and a negative number of
the same absolute value. In the context of
the formulation
\eqref{eq:tropical+linear+inequality+system}, this corresponds to
allowing the same variable $x_j$ on both sides with the same $a_{ij}$
term. Introducing {\em balanced numbers} are a standard way to carry out
such additions: $5\oplus(\ominus 5)=\bullet 5$, thus obtaining the {\em
  symmetrized semiring}. Unfortunately, it is not possible to order
this structure, and hence, it is desirable to study the geometry
restricted to the signed numbers $\TTpm$.

\subsection{Our contributions}
Tropical convexity has been well-studied for $\TTmax$ (see
\cite{AllamigeonBenchimolGaubertJoswig:2015}). 
Balanced numbers constitute an important challenge in extending this
notion to signed tropical numbers.
Consider for example the points $(1,1)$
and $(\ominus 1,\ominus 1)$ in $\TTpm^2$. It is natural to include all points
$(a,a)$ and $(\ominus a,\ominus a)$ for $a\le 1$ as well as
$(-\infty,-\infty)$ as the convex
combinations of these points. Our convexity notion will additionaly
include every point of the form $(a,b)$, $(\ominus a,\ominus b)$, $(\ominus a, b)$, and $(a, \ominus b)$
for $a,b\le 1$; thus, the convex combination of these two points in
$\TTpm^2$ will be a rectangle rather than a line segment. We obtain
these combinations by `resolving' the balanced combination $(\bullet 1,\bullet 1)$
of these two points.

Our definition of signed tropical convexity
(Definition~\ref{def:inner+hull}) just arises from the usual
definition of tropical convexity by replacing equality `$=$' with the
balance relation `$\beq$' first introduced in~\cite{Plus:90}.
We provide multiple arguments justifying why this is the right
definition. A standard way to introduce the tropical semiring is via
Puiseux lifts. It turns out (Theorem~\ref{thm:lift+hull}) that this
construction yields the tropicalization of the union of all
possible lifts. This notion of signed tropical convexity also comes up
naturally in the context of hyperoperations
(Section~\ref{subsec:hull+hyperoperation}).
Instead of balanced numbers, hyperoperations have a multi-valued
addition of signed numbers. Defining convexity for hyperoperations
coincides with our notion. However, the advantage of using balanced
numbers is to maintain a finite representation for computations.

Certain properties of the signed tropical convex hull are surprising
and, compared to usual convexity, harder to deal with. For example, in
contrast to the classical notion,
there is no unique minimal generating set of a convex set: the convex
hulls of the points $\{(1,1),(\ominus 1,\ominus 1)\}$ and of
$\{(1,\ominus 1), (\ominus 1,1)\}$ coincide.
The duality of signed tropical convex hulls and tropical linear inequality systems is reflected in the dual notions of non-negative kernel~\eqref{eq:non+negative+kernel} and open tropical cones~\eqref{eq:def+sep}. 
We formalize a new version of Farkas' lemma (Theorem~\ref{thm:farkas+lemma}) for signed tropically convex sets.
We deduce it in a geometric way from new versions of Fourier-Motzkin
elimination for signed numbers
(Theorem~\ref{thm:balanced+fourier+motzkin} and
Corollary~\ref{cor:signed+fourier+motzkin}).

There is a curious difference between open and closed signed tropical halfspaces in
our framework. 
 Whereas open halfspaces are always convex, closed
halfspaces are typically not. In fact, even signed tropical
hyperplanes are convex in very special cases only
(Example~\ref{ex:coordinate+hyperplane}).
Whereas we show that the convex hull of a set of points coincides with
the intersection of all open halfspaces containing these points
(Theorem~\ref{thm:outer+hull+open+halfspaces}), the analogous
statement is not true for  closed tropical halfspaces
(Remark~\ref{rem:closed+tropical+halfspaces+hull}).

Nevertheless, we can derive a Minkowski-Weyl theorem
(Theorem~\ref{thm:signed+Minkowski+Weyl}): for every finite set of
points, their convex hull can be obtained as the intersection of
finitely many closed halfspaces, and conversely, if the intersection
of a finite set of closed halfspaces is convex (which is not always
the case), it can be obtained as the convex hull of a finite set of
points. The proof of the first direction is based on a
version of Fourier-Motzkin elimination for non-strict
inequalities. However, the elimination procedure may
create balanced coefficients that have to be resolved by signed
numbers. Such a transformation can be easily obtained for the case of
strict inequalities (Theorem~\ref{thm:replace+balanced}), but 
becomes rather challenging in the non-strict case. In fact, 
our proof (Proposition~\ref{prop:resolving+balanced+coefficiens+teq})
only shows existence, but does not even yield a finite algorithm.

Finally, we relate our notion to the known concept of tropical
convexity  over $\Tgeq$. We show that the signed tropical convex hull
can be obtained as the union of unsigned hulls in each orthant (Theorem~\ref{thm:generators+for+all+orthants}).

\subsection{Related work}

Our notion of signed tropical convexity heavily relies on the concept of the symmetrized tropical semiring $\TSS$, which goes back to~\cite{Plus:90}, and was further developed in~\cite{AkianGaubertGuterman:2014,Rowen:2016}, among others.
Signed numbers arise in the context of tropical convexity in~\cite{AllamigeonBenchimolGaubertJoswig:2015}, however only as coefficients for an inequality system.
The technically difficult aspects are the necessary properties of equality and order relations.
While~\cite{AkianGaubertGuterman:2014} also developed different notions replacing orders or equalities, they do not provide all necessary concepts to deal with the new notion of tropical convexity.
The relations $\beq$, $\teq$ and $>$ we use also appear in the context of hyperfields in~\cite{JellScheidererYu:2018}, where images of semi-algebraic sets are studied. 
The duality of the tropical analog of polar cones in~\cite{GaubertKatz:2009} can be considered as a predecessor of our duality in Section~\ref{subsec:dual+notions}. 
Infeasibility certificates for linear inequality systems were deduced from the duality of mean payoff games in~\cite{GrigorievPodolskii:2018, AllamigeonGaubertKatz:2011}.
A tropical version of Fourier-Motzkin elimination was established in~\cite{AllamigeonFahrenbergGaubertKatz:2014}.
The latter results rely on the (tropical) non-negativity of the variables and cannot be transferred directly to our setting, as we discuss also in Remark~\ref{rem:connection+nonnegative+FM} and Remark~\ref{rem:connection+nonnegative+Farkas}. 
The tropicalizations of polytopes~\cite{DevelinYu:2007} or more general semialgebraic sets~\cite{JellScheidererYu:2018} leads to the image of a single object. However, our construction naturally leads to the tropicalization of a union of polytopes arising as the convex hull of lifts of points. 
This is in some sense dual to the representation established in~\cite{JellScheidererYu:2018}, where all satisfied equations and inequalities are needed to describe the tropicalization of a single object.
Parallel to our work, similar structures for signed numbers are developed in~\cite{AkianGaubertRowen:2019,AkianAllamigeonGaubertSergeev:2019}. 

%Intuition for symmetrized semiring from asymptotics of polynomials.

% By introducing slack variables $s \in \TTmax^n$, this is equivalent to solving a tropical equality system with the natural positivity constraint on the variables.

%% \subsection{Related algorithmic problems}

%% \begin{itemize}
%% \item Mean Payoff games
%% \item disjunctive programming
%% \item Scheduling (MILP)
%% \end{itemize}

\section{Signed numbers and orderings}\label{sec:signed-numbers}

We introduce the necessary terminology for our purposes. For a recent comprehensive introduction to signed numbers and the symmetrized semiring, see~\cite{AkianGaubertGuterman:2014}. 

\subsection{Signed numbers}

We define the signed tropical numbers $\TTpm$ by glueing two copies of $\left(\RR \cup \{-\infty\}\right)$ at $-\infty$.
One copy is declared the \emph{non-negative tropical numbers} $\Tgeq$ (this is often denoted by $\TT_{\max}$ in the literature), the other copy forms the \emph{non-positive tropical numbers} $\Tleq$.
Most of the time, we denote $-\infty$ by $\Zero$ as it is the tropical zero element.
The elements in $\Tleq \setminus \{\Zero\}$ are marked by the symbol $\ominus$.
The signed tropical numbers $\TTpm$ have a natural norm $|\,.\,|$ which maps each element of $\Tgeq$ to itself and removes the sign of an element in $\Tleq$. 
This gives rise to the order
\begin{equation} \label{eq:order+signed+numbers}
x \leq y \qquad \Leftrightarrow \qquad \begin{cases} x \in \Tleq \text{ and } y \in \Tgeq \\
  x \leq y  \text{ for } x,y \in \Tgeq \\
  |x| \geq |y|  \text{ for } x,y \in \Tleq
\end{cases} \enspace .
\end{equation}
Furthermore, we obtain the strict order $x < y \Leftrightarrow x \leq y \wedge x \neq y$.
The tropical signed space $\TTpm^d$ is the union of $2^d$ orthants which are copies of $\Tgeq^d$ glued along their boundary. 

\subsection{Balanced numbers}

To develop the technical tools for dealing with signed numbers, we use the \emph{symmetrized semiring} $\TSS$ which forms a semiring containing $\TTpm$, introduced in~\cite{Plus:90}.
This semiring is constructed with a third copy of $\RR \cup \{\Zero\}$ by glueing again at $\Zero$.
We denote the third copy, the \emph{balanced numbers}, by $\TT_{\bullet}$ and mark the elements by the symbol $\bullet$.
Unfortunately, the symmetrized semiring $\TSS$ cannot be ordered.
We extend the norm $|\,.\,|$ in such a way that it removes the $\bullet$ from an element in $\TT_{\bullet}$ and leaves the corresponding element in $\Tgeq$.
The complementary map $\tsgn$ from $\TSS$ to $\{\oplus,\ominus,\bullet,\Zero\}$ remembers only in which of the sets an elements lies: positive tropical numbers $\Tgt = \Tgeq \setminus \{\Zero\}$, negative tropical numbers $\Tlt = \Tleq \setminus \{\Zero\}$, balanced non-zero tropical numbers $\TT_{\bullet} \setminus \{\Zero\}$ or the tropical zero $\{\Zero\}$. 

Next, we define the binary operations of the semiring.
For $x,y \in \TSS$, we define the addition by
\begin{equation} \label{eq:balanced+addition}
  x \oplus y = \begin{cases}
    \argmax_{x,y}(|x|,|y|) & \text{ if } |\chi|=1 \\
    \bullet \argmax_{x,y}(|x|,|y|) & \text{ else } \enspace .
  \end{cases}
\end{equation}
where $\chi = \SetOf{\tsgn(\xi)}{\xi \in (\argmax(|x|,|y|))}$.
%% Let $x,y \in \TSS$ and set $z = \max(|x|,|y|) \in \Tgeq$. Furthermore let $\chi$ be the set of signs of the argument where the maximum of the norm is attained, that is
%% \begin{equation}
%%   \chi = \{\tsgn(\zeta) \mid \zeta \in \{x,y\} \text{ and } |\zeta| = z \} \enspace .
%% \end{equation}
%% With this, we can define the addition by
%% \begin{equation}
%%   x \oplus y = \begin{cases}
%%     z & \text{ if } \chi \subseteq \{+,\Zero\}\\
%%     \ominus z & \text{ if }  \chi = \{-\}\\
%%     \bullet z & \text{ else } \enspace .
%%   \end{cases}
%% \end{equation}
Note that we omit the sign for elements in $\Tgeq$. 
For the multiplication we set
\begin{equation}
  x \odot y = \left(\tsgn(x) * \tsgn(y)\right) (|x| + |y|) \enspace ,
\end{equation}
where the $*$-multiplication table is the usual multiplication of $\{-1,1,0\}$ for $\{\ominus,\oplus,\bullet\}$ with the additional specialty that multiplication with $\Zero$ yields $\Zero$. 
%% \begin{tabular}{c|cccc}
%%   & + & - & $\bullet$ & $\Zero$ \\ \hline
%%   + & + & - & $\bullet$ & $\Zero$ \\ 
%%   - & - & + & $\bullet$ & $\Zero$ \\
%%   $\bullet$ & $\bullet$ & $\bullet$ & $\bullet$ & $\Zero$ \\
%%   $\Zero$ & $\Zero$ & $\Zero$ & $\Zero$ & $\Zero$ 
%% \end{tabular}

The operations $\oplus$ and $\odot$ extend to vectors and matrices componentwise.
Observe that the operations agree with the usual $\max$-tropical operations on $\Tgeq$. 

We can also consider $\ominus$ as a unary selfmap of the semiring; to this extent, we set
\[
  \ominus x =
\begin{cases}
  \ominus x & \text{ if } x \in \Tgt \\
  |x| & \text{ if } x \in \Tlt \\
  x & \text{ if } x \in \Tzero
\end{cases}
\enspace .
\]
The map $\ominus \colon \TSS \to \TSS$ is a semiring homomorphism. 
%% \end{lemma}
%% \begin{proof}
%%   LOOK FOR REFERENCE OR CHECK AXIOMS
%% \end{proof}
In particular, this justifies to write $a \ominus b$ for $a \oplus (\ominus b)$.

Furthermore, the absolute value fulfills $|a \oplus b| = |a| \oplus |b|$ by definition of the addition. 

\begin{example}
  Using the definitions, we see that $-5$ is positive, $\ominus 6$ and $\ominus -6$ are negative, $\bullet 3$ is balanced.
  Furthermore, the absolute value of $-5$ is $|-5| = -5$, of $\ominus 6$ is $|\ominus 6| = 6$, and $|\bullet 3| = 3$.
  Some simple sums are $3 \oplus (\ominus 3) = \bullet 3$, $-3 \oplus 5 = 5$, $-3 \oplus (\ominus 5) = \ominus 5$, $\bullet 2 \oplus 4 = 4$, $\bullet -3 \oplus \ominus -5 = \bullet -3$.
  Finally some simple products are $\bullet 3 \odot 5 = \bullet 8$, $\ominus 4 \odot -6 = \ominus -2$, $\ominus 1 \odot \ominus 1 = 2$, $\bullet 3 \odot \Zero = \Zero$, $\ominus 4 \odot \Zero = \Zero$. 
\end{example}

\subsection{Extending the order}

As already mentioned, the semiring $\TSS$ cannot be ordered in a consistent way with respect to its binary operations.
However, we will equip it with some binary relations, which partly fulfill the tasks of an order.
They occur under a different terminology in~\cite{JellScheidererYu:2018}; see~\ref{subsec:hull+hyperoperation}.

\subsubsection{Signed order}

Even if $\TSS$ cannot be ordered totally, we can extend the ordering from $\TTpm$ partially by setting
\begin{equation} \label{eq:extended+partial+order+simple}
  x > y \quad \Leftrightarrow \quad x \ominus y \in \Tgt \enspace . %x \ominus y > \Zero \enspace .
\end{equation}

This is equivalent to
\begin{equation} \label{eq:extended+partial+order}
  x > y  \quad \Leftrightarrow \quad
  \begin{cases}
  x > y & \text{ for } x,y \in \TTpm, \text{ see~\eqref{eq:order+signed+numbers} } \\
  x > |y| &  \text{ for } x \in \TTpm, y \in \Tzero \\
  \ominus |x| > y & \text{ with } x \in \Tzero, y \in \TTpm 
\end{cases} \enspace .
\end{equation}
Note that there are pairs in $\TTpm \times \Tzero$ and in $\Tzero \times \TTpm$ which are not comparable.
In particular, the signed numbers
\[
\{x \in \TTpm \mid x \not< a \text{ and } x \not> a\} \enspace ,
\]
which are incomparable to $a \in \Tzero$ via '$<$', form the interval
\begin{equation} \label{eq:incomparable+elements}
\Uncomp(a) := \left[\ominus |a|, |a|\right] := \{ x \in \TTpm \mid \ominus |a| \leq x \leq |a| \} \enspace .
\end{equation}
We also denote the set incomparable to a signed element $a \in \TTpm$, which is only the singleton $\{a\}$, by $\Uncomp(a)$.
 We extend this to vectors by setting $\Uncomp(v) = \prod_{i \in [d]} \Uncomp(v_i)$. 
Note that also no pair in $\Tzero \times \Tzero$ is comparable. 

The relation~\eqref{eq:extended+partial+order} gives rise to a non-strict relation
\begin{equation} \label{eq:extended+partial+order+non+strict}
  x \geq y \qquad \Leftrightarrow \qquad x > y \text{ or } x = y \enspace .
\end{equation}
which turns out to be a partial order in Corollary~\ref{cor:non-strict-relation-transitive}.

Observe that the ordering is compatible with the reflection map, in the sense that
\begin{equation} \label{eq:symmetry+reflection+order}
  x \geq y \qquad \Leftrightarrow \qquad  \ominus y \geq \ominus x \enspace .
\end{equation}

A useful property of strict inequalities is that they can be added together.
\begin{lemma} \label{lem:adding+strict+inequalities}
  For $a,b,c,d \in \TSS$, we have the implication
  \begin{equation}
    a < b \text{ and } c < d  \qquad \Rightarrow \qquad a \oplus c < b \oplus d \enspace .
  \end{equation}
\end{lemma}
\begin{proof}
  By definition, we first get $b \ominus a > \Zero$ and $d \ominus c > \Zero$.
  As addition is closed in $\Tgt$, this yields $b \ominus a \oplus d \ominus c > \Zero$.
  The claim follows from~\eqref{eq:extended+partial+order+simple}.
\end{proof}

\begin{remark} \label{rem:reformulate+inequality+order}
 In general, the strict and non-strict partial order `$<$' and
 `$\leq$' on $\TSS$ is not compatible with addition. The inequality $3
 < 4$ does not imply $3 \oplus 5 < 4 \oplus 5$, and $3 \leq 4$ does
 not imply $\ominus 4 = 3 \ominus 4 \leq 4 \ominus 4 = \bullet
 4$. This is the main motivation for introducing the relation `$\teq$'
 below, which is not an ordering (as it lacks transitivity) but it is compatible with the addition.
  
An advantage of strict inequalities is the validity of
  \[
  a \oplus b > c \Leftrightarrow a > c \ominus b \enspace .
  \]
  The analogous reformulation
  \[
  a \oplus b \geq c \Leftrightarrow a \geq c \ominus b \enspace .
  \]
  is wrong in general.
  For example, $2 \oplus 5 \geq 5$ but $2$ is incomparable with $5 \ominus 5 = \bullet 5$.
%    This is why weak tropical duality fails, see~\ref{sec:strong+weak+tropical+duality}.
  However, such reformulations hold for the relation `$\teq$', which we show in Lemma~\ref{lem:basic+properties+teq}\eqref{item:teq+other+side}. 
\end{remark}

\subsubsection{Balanced relations} \label{subsubsec:non-strict-relation}

The balance relation '$\Delta$' was introduced in~\cite{Plus:90}; we
will use the notation $\beq$ in this paper. We define 
\[
x\beq y \quad \Leftrightarrow \quad x \ominus y \in \Tzero \enspace .
\]
The following characterizations are immediate from the definitions.
For more properties of $\beq$, we refer to~\cite[{\S IV}]{Plus:90}.
\begin{lemma} \label{lem:basic+properties+beq} Let $a, b \in \TSS$.
  \begin{enumerate}[(a)]
\item\label{item:beq+char}
  $a \beq b$ is equivalent to $\left(a \in \Tzero, |a| \geq |b|\right) \vee \left(b \in \Tzero, |b| \geq |a|\right)\vee \left(a=b\right)$.
\item\label{item:beq+U} If $b\in \TTpm$, then $a\beq b$ is equivalent to
  $b\in \Uncomp(a)$.
\end{enumerate}
\end{lemma}

\begin{remark}\label{rem:not+equivalence}
  Note that 
$\beq$ is {\em not} an equivalence relation, as the example
  \[
  1 \beq \bullet 6, \quad \bullet 6 \beq 3, \quad \text{ but } 1 \not\beq 3
  \]
  shows.
\end{remark}

We introduce the binary relation
\begin{equation}
  x \teq y \qquad \Leftrightarrow \quad   x > y \text{ or } x \beq y\quad
\Leftrightarrow  \qquad x \ominus y \in \Tgeq \cup \Tzero \enspace .
\end{equation}
Note that $a \beq b$ is equivalent to $(a\teq b)\vee (a\seq
b)$. Remark~\ref{rem:not+equivalence} shows that $\teq$ is {\em not} a
partial order.

Recall from Remark~\ref{rem:reformulate+inequality+order} that bringing terms to the other side of a non-strict inequality with `$\geq$' is not valid in general. 
The next lemma shows, among other simple properties, that `$\teq$' is
compatible with the semiring operations\iffull.\else; the proof is deferred to the Appendix.\fi

\goodbreak
\begin{restatable}{lemma}{bpteq} \label{lem:basic+properties+teq}
  Let $a,b,c,d \in \TSS$.
  \begin{enumerate}[(a)]
  \item \label{item:teq+other+side} $a \oplus c \teq b \Leftrightarrow a \teq b \ominus c$
  \item \label{item:teq+summing} $a \teq b \quad \wedge \quad c \teq d \quad \Rightarrow \quad a \oplus c \teq b \oplus d$.
  \item \label{item:restricted+signed+transitivity} If $c \in \TTpm$, then $b \teq c$ and $c \teq a$ imply $b \teq a$.
    %  \item $\bullet a \oplus b \teq \Zero \quad \Leftrightarrow |a| \oplus b \teq \Zero \wedge \ominus |a| \oplus b \teq \Zero$
  \item \label{item:teq+affine+monotony} $a \teq b$ implies $c \odot a
    \oplus d \teq c \odot b \oplus d$ for $c \in \Tgeq$ and $c \odot b
    \oplus d \teq c \odot a \oplus d$ for $c \in \Tleq$.
  \end{enumerate}
\end{restatable}

\newcommand{\proofbpteq}{\begin{proof}
  \eqref{item:teq+other+side} The claim follows directly from the definition and the properties of the semiring $\TSS$. 

  \smallskip

  \eqref{item:teq+summing} Using the definition, we obtain $a \ominus b, c \ominus d \in \Tgeq \cup \Tzero$. This implies already $a \oplus c \ominus b \ominus d \in \Tgeq \cup \Tzero$.

  \smallskip
  
  \eqref{item:restricted+signed+transitivity} 
For a contradiction, assume $b\not\teq a$, that is, $b\ominus a\in
\Tlt$.

\noindent{\sl Case I}: $|b|>|a|$. In this case, $b\in \Tlt$. 
 Since $b\ominus c\in \Tgeq \cup
\Tzero$, it follows that $c\in \Tlt$ and $|c|\ge |b|$, using $c \in \TTpm$. We now get a
contradiction to $c\ominus a\in \Tgeq \cup
\Tzero$, since $|c|\ge |b|>|a|$.

\noindent{\sl Case II}: $|a|>|b|$. This case follows by an analogous
argument. With $a\in \Tgt$, the condition $c\ominus a\in \Tgeq \cup
\Tzero$ implies $c\in \Tgt$ and $|c|\ge |a|>|b|$. This contradicts
$b\ominus c\in \Tgeq \cup
\Tzero$.

\noindent{\sl Case III}: $|a|=|b|$. In this case, we must have $b\in
\Tlt$ and $a\in \Tgt$. %, $a=\ominus b$.
We thus obtain  $c\in \Tlt$ as
in case I, but also $c\in \Tgt$ as in case II, a contradiction.

  \smallskip

  \eqref{item:teq+affine+monotony} The expression $c \odot (a \ominus
  b) = c \odot a \ominus c \odot b$ is in $\Tgeq$ for $c \in \Tgeq$
  and in $\Tleq$ for $c \in \Tleq$. Now, the statement follows
  from~\eqref{item:teq+summing} with $d \teq d$.

 \end{proof}
}
\iffull\proofbpteq
\fi

\section{Tropical convexity of signed numbers}

\subsection{Signed tropical convex combinations}
Let us recall the notation that for a matrix $A\in \Tgeq^{d \times n}$,
and a vector $x\in \Tgeq^n$, we denote by 
$A\odot x\in \Tgeq^n$ the tropical matrix product.
The \emph{tropical convex hull} $\tconv(A)$ of the columns of a matrix
$A\in  \Tgeq^{d \times n}$, studied in~\cite{BriecHorvath:2004,CohenGaubertQuadrat:2005,DevelinSturmfels:2004}, is defined as
\begin{equation} \label{eq:positive+convex+hull}
 \tconv(A) = \SetOf{A \odot x}{ x\in \Tgeq^n, \bigoplus_{j \in [n]} x_j = 0 } \subseteq \Tgeq^d
\enspace .
\end{equation}
%
%This can be extended to not-necessarily finite sets by taking the union over the convex hull of all finite subsets.
In this definition it is essential that all columns of $A$ lie in the
non-negative orthant $\Tgeq^d$. For general matrices in
$\TTpm^{d\times n}$, the product $A\odot x$ may contain balanced entries.
We now extend the notion of the tropical convex hull to
$\TTpm^d$.
Note that we switch freely between a matrix and its set of columns. 
%
%The motivation for our generalization comes from the connection with the feasibility of tropical linear inequality systems, which is established in Corollary~\ref{cor:zero+in+convex+hull}, noting that $\ker^+(A)$ just encodes the homogeneous version of~\eqref{eq:tropical+linear+inequality+system}. 
%

\begin{figure} 
  \begin{tikzpicture}

  \Koordinatenkreuz{-2.1}{2.1}{-1.2}{2.3}{$x_1$}{$x_2$}

  \coordinate (A) at (1.4,1.4){};
  \coordinate (B) at (-1.7,-0.9){};
  \coordinate (C) at (-1.0,-0.5){};
  \coordinate (AB) at (-1.7,1.4){};
  \coordinate (AC1) at (-1.0,1.0){};
  \coordinate (AC2) at (1.0,1.0){};
  \coordinate (BC) at (-1.3,-0.5){};

  \draw[fill,Linesegment,opacity=0.4] (A) -- (AB) -- (B) -- (BC) -- (C) -- (AC1) -- (AC2) -- cycle;

  \node[Endpoint] at (A){};
  \node[Endpoint] at (B){};
  \node[Endpoint] at (C){};
   
\end{tikzpicture}
  \caption{The signed tropically convex hull of $\{(3,3),(\ominus 1, \ominus 0),(\ominus 4, \ominus 2)\}$. We omit labels for the axes as the origin is $(-\infty,-\infty)$ and therefore infinitely far away. }
  \label{fig:first+hull}
\end{figure}
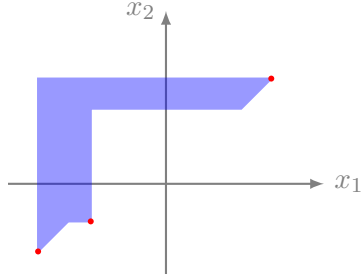

\begin{definition}[Inner hull] \label{def:inner+hull}
  The \emph{(signed) tropical convex hull} of the columns of the
  matrix $A \in \TTpm^{d \times n}$ is defined as
  \begin{equation} \label{eq:hull+union+intervals}
  \tconv(A) =\SetOf{z\in \TTpm^d}{z\beq A\odot x, x\in \Tgeq^n, \bigoplus_{j \in [n]} x_j = 0 } \subseteq \TTpm^d\enspace .
  \end{equation}
  Such a set is a \emph{(signed) tropical polytope}. 
  The tropical convex hull of an arbitrary set $M \subseteq \TTpm^d$ is the union
  \[
  \tconv(M) = \bigcup_{V \subseteq M, V \text{ finite }} \tconv(V) \enspace .
  \]
A subset $M \subseteq \TTpm^d$ is \emph{tropically convex} if
$M=\tconv(M)$.
\end{definition}
This hull construction generalizes~\eqref{eq:positive+convex+hull}
because if $A\in \Tgeq^{d\times n}$ then $A\odot x\in \TTpm^d$. In this case, Lemma~\ref{lem:basic+properties+beq}\eqref{item:beq+char} implies that $z\beq
A\odot x$ holds only for $z=A\odot x$.

Using Lemma~\ref{lem:basic+properties+beq}\eqref{item:beq+U}, we can
write \eqref{eq:hull+union+intervals} equivalently as 
\begin{equation}\label{eq:hull+U}
   \tconv(A) =\bigcup \SetOf{\Uncomp(A\odot x)}{ x\in \Tgeq^n, \bigoplus_{j \in [n]} x_j = 0 } \subseteq \TTpm^d\enspace .
  \end{equation}

\begin{example} \label{ex:first+convex+hull+critical+points}
  The combinations of the tropical convex hull depicted in Figure~\ref{fig:first+hull}, where balanced numbers occur, can be calculated via
  \begin{equation*}
    \begin{aligned}
      (-3) \odot \pvec{3}{3} \oplus \pvec{\ominus 1}{\ominus 0} &= \pvec{\ominus 1}{\bullet 0} \;,\qquad 
      (-2) \odot \pvec{3}{3} \oplus \pvec{\ominus 1}{\ominus 0} &=& \pvec{\bullet 1}{1} \\
      \pvec{3}{3} \oplus (-1) \odot \pvec{\ominus 4}{\ominus 2} &= \pvec{\bullet 3}{3} \;, \qquad
      (-1) \odot \pvec{3}{3} \oplus \pvec{\ominus 4}{\ominus 2} &=& \pvec{\ominus 4}{\bullet 2} \enspace .
    \end{aligned}
  \end{equation*}
      A more precise way, how these points can be used to determine the signed tropical convex hull, via the tropical convex hull of the intersection with each orthant is given in Theorem~\ref{thm:generators+for+all+orthants}. 
\end{example}

\begin{remark}
  There is no unique minimal generating set in the usual sense as one can see from $\tconv((0,0),(\ominus 0, \ominus 0)) = \tconv((0,\ominus 0),(\ominus 0, 0))$.
\end{remark}

We now derive some elementary properties of this convexity notion. The
following are immediate from the definition, as~\eqref{eq:hull+U} is just a componentwise construction. 

\begin{proposition} \label{prop:immediate+properties+convex+sets} \quad
  \begin{enumerate}[(a)] 
  \item \label{item:intersection+convex+sets}
  The intersection of tropically convex sets is tropically convex.
  \item \label{item:projection+convex+sets}
  The coordinate projection of tropically convex sets is tropically convex.
  \end{enumerate}
\end{proposition}
    
Next, we show that convexity follows already by showing the
containment of line segments
(Proposition~\ref{prop:generation+by+line+segments}), and that 
$\tconv(.)$ is a closure operator, i.e., 
the convex hull of a set is a tropically convex set
(Proposition~\ref{prop:convex+hull+closure}). The following technical
lemma will be needed for these proofs. \iffull\else The next three
proofs are deferred to the Appendix.\fi

\begin{restatable}{lemma}{beqcomb}\label{lem:beq+comb} \quad
\begin{enumerate}[(a)]
\item \label{item:u+reduce} Let $a\in \TSS$, $b\in \TTpm$, and $z\in \Uncomp(a\oplus
  b)$. Then there exists an $a'\in \Uncomp(a)$ such that $z\in
  \Uncomp(a'\oplus b)$.
 \item \label{item:u+contain} 
If $a \in \Uncomp(x)$, $b\in \Uncomp(y)$, and $c\in \TTpm$, then
$\Uncomp(c\odot a\oplus b)\subseteq \Uncomp(c\odot x\oplus y)$.
\end{enumerate}
\end{restatable}
\newcommand{\proofbeqcomb}{
\begin{proof}
\eqref{item:u+reduce} If $a\in \TTpm$, then $a'=a$ satisfies the requirements.
For the rest of the proof, we assume $a\in \Tzero$.
If $|b|>|a|$, then we can set $a'=|a|$.
In this case, $a'\oplus b=a\oplus b=b\in \TTpm$.
Consider now the case $|a|\ge |b|$, which implies $a\oplus b=a$.
Then $z\in \Uncomp(a \oplus b)$ if and only if $|z|\le |a|$.
For $|a|\ge |z| > |b|$, we set $a'=z$.
If $|a|\ge |b|\ge |z|$, then we set $a'=\ominus b$. In both
  cases it is easy to see that  $z\in
  \Uncomp(a'\oplus b)$.

\smallskip
\eqref{item:u+contain} Note that
$|a|\le |x|$ and $|b|\le |y|$, and consequently, $|c\odot a\oplus
b|\le |c\odot x\oplus y|$. If $c\odot x\oplus y$ is balanced, then the
claim follows: $\Uncomp(c\odot x\oplus y)$ contains all $r\in
\TTpm$ with $|r|\le |c\odot x\oplus y|$; this holds for all  $r\in\Uncomp(c\odot a\oplus
b)$.

Hence, assume that $c\odot x\oplus y$ is not balanced.
In particular, $x$ or $y$ is not balanced. If both $x,y\in \TTpm$, then $a=x$
and $b=y$ and thus the claim is immediate. The remaining case is when
exactly one of $x$ and $y$ is balanced.
Let us  assume $y\in \TTpm$;
the case $x\in \TTpm$ follows similarly. Now we have $b=y$, and we
must also have $|y|>|c\odot x|$ as otherwise $c\odot x\oplus y$ would
be balanced. Consequently, $c\odot x\oplus y=y$. On the other hand,
$|a|\le |x|$ and $b=y$ imply $|c\odot a|<|b|$, and therefore $c\odot a\oplus
b=y$, and the claim follows.
\end{proof}}

\iffull\proofbeqcomb\fi

\begin{restatable}{proposition}{genlineseg} \label{prop:generation+by+line+segments}
  An arbitrary subset $M \subseteq \TTpm^d$ is tropically convex if and only if $\tconv(\{p,q\}) \subseteq M$ for all $p,q \in M$. 
\end{restatable}
\newcommand{\proofgenlineseg}{
\begin{proof}
  For a tropically convex set, the tropical convex hull of all two-element subsets is contained by definition.
In the converse direction, we show by induction on $n$ that if we
select any $n$ vectors from $M$ as the columns of a matrix $A\in \TTpm^{d\times
  n}$, then $\Uncomp(A\odot x)\subseteq M$ for any $x\in \Tgeq^n$, $\bigoplus_{j \in [n]} x_j = 0$.
 The case $n=2$ follows by the
assumption; consider now $n\ge 3$ and assume that the claim holds for $n-1$.

Let $z\in \Uncomp(A\odot x)$.
  Without loss of generality, we can assume that $x_1=0$. We set $s
  =\bigoplus_{\ell = 1}^{n-1} x_\ell \odot a^{(\ell)} \in \TSS^d$,
  where $a^{(\ell)}$ is the $\ell$-th column of $A$. We let $q =  a^{(n)}$.
Then, $A\odot x=s\oplus x_n\odot q$. 

We can apply Lemma~\ref{lem:beq+comb}\eqref{item:u+reduce} to each
component of $z$, $s$, and $x_n\odot q$. Thus, we obtain a vector
$p\in \Uncomp(s)$ such that $z\in \Uncomp(p\oplus x_n\odot q)$. By
induction, $p\in M$, and thus $z\in \tconv(\{p,q\})\subseteq M$ by the
assumption. This completes the proof.
\end{proof}}
\iffull\proofgenlineseg\fi

\begin{restatable}{proposition}{conhc}\label{prop:convex+hull+closure}
For any matrix $A\in \TTpm^{d\times n}$, the convex hull
$\tconv(A)$ is tropically convex. Consequently, $\tconv(\tconv(A))=\tconv(A)$. 
\end{restatable}
\newcommand{\proofconhc}{
\begin{proof}
Using Proposition~\ref{prop:generation+by+line+segments}, it suffices
to show that if $p,q\in \tconv(A)$, $\lambda\in \Tgeq$, $\lambda\le
0$, then $\Uncomp(p\oplus \lambda\odot q)\subseteq \tconv(A)$. 

Let $x,y\in \Tgeq^n$, $\bigoplus_{j \in [n]} x_j = 0$, $\bigoplus_{j
  \in [n]} y_j = 0$ such that $p\in\Uncomp(A\odot x)$ and $q\in\Uncomp(A\odot y)$.
We let $z=x\oplus\lambda\odot q$; clearly, $z\in \Tgeq^n$ and
$\bigoplus_{j \in [n]} z_j = 0$.
From 
Lemma~\ref{lem:beq+comb}\eqref{item:u+contain}, we obtain
that 
 \[\Uncomp(p\oplus \lambda\odot q)\subseteq \Uncomp\left((A\odot x)\oplus \lambda \odot
(A\odot y)\right)= \Uncomp(A\odot z)\subseteq \tconv(A)\enspace .\]
\end{proof}
}
\iffull\proofconhc\fi

%% \begin{example} \label{ex:tropical+polytopes+convex}  
%% \end{example}

\begin{example} \label{ex:coordinate+hyperplane}
 A {\em (signed) tropical hyperplane} is of the form
  \[
  \Hyp(a) = \SetOf{x \in \TTpm^d}{a \odot x \in \Tzero^d} \enspace .
  \]
It is easy to see that this set is tropically convex if
$|\supp(a)|=1$, where $\supp(a) = \SetOf{i \in [d]}{a_i \neq \Zero}$. Therefore, using
Proposition~\ref{prop:immediate+properties+convex+sets}\eqref{item:intersection+convex+sets}, 
we see that 
for a subset $I \subseteq [d]$ and a point $b \in \TTpm^d$, the set 
  \[
  \SetOf{z \in \TTpm^d}{z_i = b_i \text{ for all } i \in I}
  \]
  is tropically convex.

  On the other hand, if $|\supp(a)|>1$, then $\Hyp(a)$ is {\em not}
tropically convex.
 To see this, let us assume that $\supp(a) \supseteq \{1,2\}$. Then $p = (\ominus a_2, a_1,\Zero,\ldots,\Zero), q = (a_2, \ominus a_1,\Zero,\ldots,\Zero) \in \Hyp(a)$.
  Because of $p \oplus q = (\bullet a_2,\bullet a_1, \Zero,\ldots,\Zero)$, the point $(a_2,a_1, \Zero,\ldots,\Zero)$ is contained in $\tconv(p,q)$.
  However, it is not an element of $\Hyp(a)$.
\end{example}

\begin{example} \label{ex:open+halfspace+convex}
  For a vector $(a_0,a_1,\dots,a_d) \in \TTpm^{d+1}$ we define the \emph{open signed (affine) tropical halfspace}
\begin{equation}
  \HC^+(a) = \SetOf{x \in \TTpm^d}{a \odot \begin{pmatrix} 0 \\ x \end{pmatrix} > \Zero} \enspace .
\end{equation}
An open signed tropical halfspace is tropically convex.
Let $c \in \TTpm^d, c_0 \in \TTpm$, $p,q \in \TTpm^d$ and $\lambda,\mu \in \Tleq$ with $\lambda \oplus \mu = 0$. For $p$ and $q$ contained in the halfspace, we have $c \odot p \oplus c_0 > \Zero$ and $c \odot q \oplus c_0 > \Zero$, and by Lemma~\ref{lem:adding+strict+inequalities},
\begin{equation} \label{eq:example+sum+strict+inequalities}
  c \odot (\lambda \odot p \oplus \mu \odot q) \oplus c_0 = \lambda \odot (c \odot p \oplus c_0) \oplus \mu \odot (c \odot q \oplus c_0) > \Zero \enspace .
  \end{equation}
If $\lambda \odot p \oplus \mu \odot q$ has a balanced component $b \in \Tzero$, then the value of $c \odot (\lambda \odot p \oplus \mu \odot q) \oplus c_0$ cannot depend on this component as it is positive. 
Hence, we can replace that component by an element in $\Uncomp(b)$ and preserve the inequality~\eqref{eq:example+sum+strict+inequalities}. 
%, we deduce from Lemma~\ref{lem:resolve+balanced} that 
\end{example}

\begin{remark} \label{rem:transformation+group}
Let $\mathbb{P} \subset \TTpm^{d \times d}$ be the set of permutation
matrices with $0$ as one and $\Zero$ as zero, and let $\mathbb{D}
\subset \TTpm^{d \times d}$ be the set of matrices with diagonal
entries from $\TTpm$ and $\Zero$ else. Their union generates the
multiplicative \emph{group of signed tropical transformations}
$\STT$. This group is the natural group of transformations which
leaves the combinatorial structure of a subset of $\TTpm^d$
unchanged. It provides a useful tool to simplify technical constructions. 
\end{remark}

\begin{example}
  We want to describe the line segment $\tconv(p,q)$ for two points $p,q \in \TTpm^d$.
  By suitable scaling with elements from $\STT$, we can assume that $p = (0,\dots,0)$, and that the entries of $q$ are ordered by increasing absolute value. 

  Analogous to the description in~\cite{DevelinSturmfels:2004}, one obtains a piecewise-linear structure where the breakpoints are determined by the absolute values of the components of $q$.
  As an additional phenomenon, the line segments flip to another orthant at each tropically negative entry of $q$.
  If the sign changes in $\ell$ coordinates at once, the line segment has dimension $\ell$. 
  We visualize several examples for the two-dimensional case in Figure~\ref{fig:line+segments}. 
  
%similar to~\cite[Proposition 4.2.2]{Briec:2015}. }
\end{example}

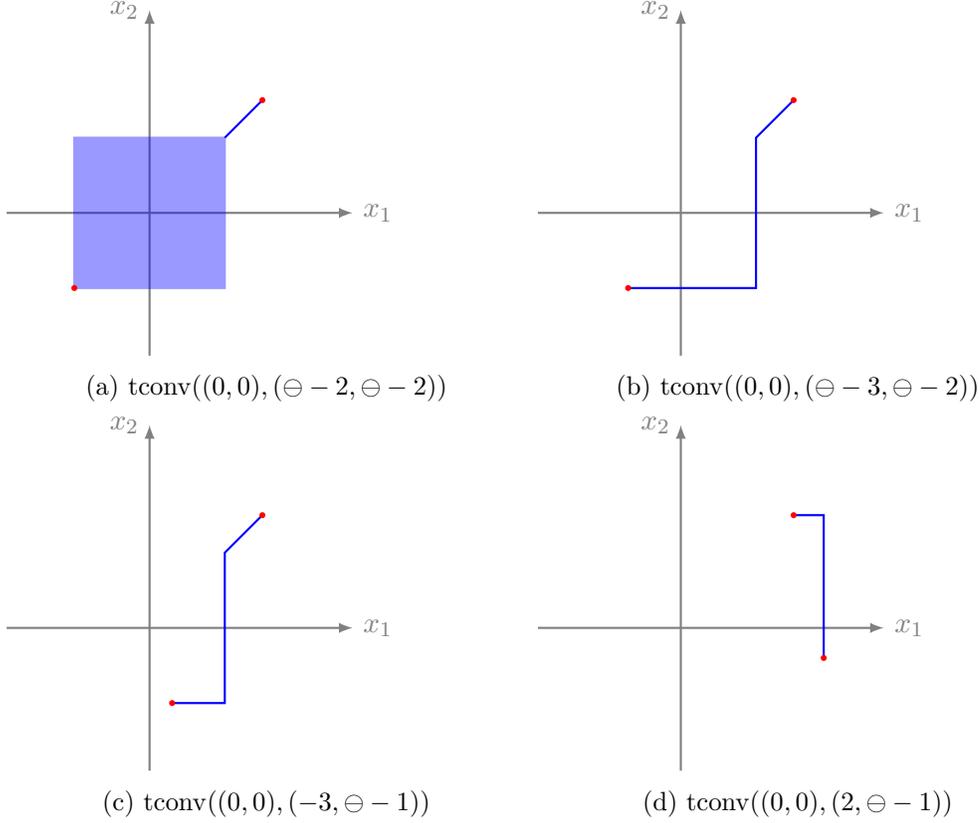
\begin{figure} 
  \newcommand\xmin{-1.9}
  \newcommand\xmax{2.7}
  \newcommand\ymin{-1.9}
  \newcommand\ymax{2.7}
  \begin{subfigure}[h]{0.42\textwidth}
    \begin{tikzpicture}

  \Koordinatenkreuz{\xmin}{\xmax}{\ymin}{\ymax}{$x_1$}{$x_2$}

  \coordinate (center) at (1.5,1.5){};
  \coordinate (pp) at (1,1){};
  \coordinate (mp) at (-1,1){};
  \coordinate (pm) at (1,-1){};
  \coordinate (mm) at (-1,-1){};
  
  \draw[Linesegment] (center) -- (pp);
  \draw[fill,Linesegment,opacity=0.4] (pp) -- (mp) -- (mm) -- (pm) -- cycle;

  \node[Endpoint] at (center){};
  \node[Endpoint] at (mm){};
   
\end{tikzpicture}
    \caption{$\tconv((0,0),(\ominus -2, \ominus -2))$}
    \label{subfig:tropical+exploded+line+segment}
  \end{subfigure}
  \begin{subfigure}[h]{0.42\textwidth}
    \begin{tikzpicture}

  \Koordinatenkreuz{\xmin}{\xmax}{\ymin}{\ymax}{$x_1$}{$x_2$}

  \coordinate (center) at (1.5,1.5){};
  \coordinate (pp) at (1,1){};
  \coordinate (pm) at (1,-1){};
  \coordinate (mm) at (-0.7,-1){};
  
  \draw[Linesegment] (center) -- (pp) -- (pm) -- (mm);

  \node[Endpoint] at (center){};
  \node[Endpoint] at (mm){};
  
\end{tikzpicture}
    \caption{$\tconv((0,0),(\ominus -3, \ominus -2))$}
  \end{subfigure}
  \begin{subfigure}[h]{0.42\textwidth}
    \begin{tikzpicture}

  \Koordinatenkreuz{\xmin}{\xmax}{\ymin}{\ymax}{$x_1$}{$x_2$}
  
  \coordinate (center) at (1.5,1.5){};
  \coordinate (pp) at (1,1){};
  \coordinate (pm) at (1,-1){};
  \coordinate (pm2) at (0.3,-1){};
  
  \draw[Linesegment] (center) -- (pp) -- (pm) -- (pm2);

  \node[Endpoint] at (center){};
  \node[Endpoint] at (pm2){};
  
\end{tikzpicture}
        \caption{$\tconv((0,0),(-3, \ominus -1))$}
  \end{subfigure}
  \begin{subfigure}[h]{0.42\textwidth}
    \begin{tikzpicture}

  \Koordinatenkreuz{\xmin}{\xmax}{\ymin}{\ymax}{$x_1$}{$x_2$}

  \coordinate (center) at (1.5,1.5){};
  \coordinate (pp) at (1.9,1.5){};
  \coordinate (pm) at (1.9,-0.4){};
  
  \draw[Linesegment] (center) -- (pp) -- (pm);

  \node[Endpoint] at (center){};
  \node[Endpoint] at (pm){};
  
\end{tikzpicture}
        \caption{$\tconv((0,0),(2, \ominus -1))$}
  \end{subfigure}
  \caption{Several tropical line segments in the plane}
  \label{fig:line+segments}
\end{figure}

%% \begin{lemma}
%%   The arbitrary intersection of tropically convex sets is tropically convex. 
%% \end{lemma}
%% \begin{proof}
%%   Let $I$ be the index set for a family $(M_i)_{i \in I}$ of tropically convex sets. Let $S$ be a finite subset of the intersection $\bigcap_{i \in I} M_i$. Then $S$ is a subset 
%% \end{proof}

\begin{definition}[Conic hull] 
  The \emph{signed tropical conic hull} of the columns of $A$ is
  \begin{equation} %\label{eq:hull+union+intervals}
    \tcone{A} = \bigcup_{\lambda \in \Tgeq^n} \Uncomp(A \odot \lambda) \enspace .
  \end{equation}
\end{definition}

The definition together with Proposition~\ref{prop:convex+hull+closure} yields the following. 

\begin{corollary}
  The conic hull of a subset of $\TTpm^d$ is tropically convex. 
\end{corollary}

\subsection{Image of Puiseux lifts} \label{subsec:puiseux+lifts}

The aim of this section is to relate our concept of convexity over $\TTpm$ to convexity over $\RR$. To achieve this, we move to another ordered field, the field of real Puiseux series $\KK = \pseries{\RR}{t}$.
This has proven to be a helpful concept in the study of tropical numbers with signs, see~\cite{Viro:2010,AllamigeonBenchimolGaubertJoswig:2015,JellScheidererYu:2018}.
It is formed by formal Laurent series with exponents in $\RR$ and coefficients in $\RR$. 
The exponent sequence is strictly decreasing and it has no accumulation point. 
This ordered field is equipped with a non-archimedean valuation $\val$ which maps all non-zero elements to their leading exponent and zero to $\Zero = -\infty$.
Additionally, the map $\sgn \colon \KK \to \{\ominus,\Zero,\oplus\}$ yields the sign of an element.
 This gives rise to the \emph{signed valuation} $\sval \colon \KK \to \TTpm$ which maps an element $k \in \KK$ to $\sgn(k) \val(k)$.
It is enough to think of Puiseux series as polynomials in $t$ with arbitrary exponents and coefficients in $\RR$. 

The \emph{tropicalization} of structured sets over $\KK$, i.e., the study of the image of a subset of $\KK^d$ is a technique which is widely used in tropical geometry.
We introduced a concept purely on the tropical side.
We will see in Theorem~\ref{thm:lift+hull}, that signed tropically convex sets are not the image of the valuation of a single convex hull but of a whole union, ranging over the fibers of tropical points.

In some sense, this is complementary to the main result in~\cite{JellScheidererYu:2018}.
While they consider semialgebraic sets over $\KK$ in general, polytopes, i.e., the convex hull of finitely many points in $\KK^d$, can be considered as a special case.
They show that one has to tropicalize all semialgebraic relations fulfilled by a set to describe its image under the signed valuation map.

Recall that for our concept of tropical convexity over $\TTpm$ the image of a single polytope under the signed valuation may not be tropically convex as the Example~\ref{ex:different+convexities} shows.
It is subject to further work to study the special case of polytopes (as semialgebraic sets) from~\cite{JellScheidererYu:2018} and to see which properties such a notion of signed tropical polytopes could provide.

%\todo[inline]{GL: tropicalization of convex sets (analogous to the tropicalization of semialgebraic sets in) versus signed tropical convex hull. Alternative approach: tropically convex sets as signed valuation of convex sets. But our approach gives nice duality. }

%We have to compare with~\cite{AllamigeonGaubertSkomra:2016} and with~\cite{JellScheidererYu:2018}. 

%For the sake of being explicit, we restrict to the field $\KK = \RR\{\{t\}\}$.

Note that the next statement is valid for more general fields with a
non-trivial non-archimedean valuation $\val$ which is surjective onto
$\Tgeq$. \iffull\else The proof is given in the Appendix.\fi

%% \begin{lemma}
%%   Let $\lambda \in \KK_{\geq 0}$ with $\sum_{j = 1}^{n} \lambda_i = 1$ and $k_1,\ldots,k_n \in \KK$. Then $\sval(\sum_{\ell=1}^{n} \lambda_{\ell} k_{\ell}) \in [\ominus ,]$.
%% \end{lemma}

\begin{restatable}{theorem}{lifthull} \label{thm:lift+hull}
  The signed hull $\tconv(A)$ is the union of the signed valuations for all possible lifts
  \[
  \tconv{A} = \bigcup_{\sval(\mathbf{A}) = A} \sval(\conv(\mathbf{A})) \enspace .
  \]
\end{restatable}
\newcommand{\prooflifthull}{
\begin{proof}
  We start with the inclusion `$\supseteq$'. 
  Let $A \in \TTpm^{d \times n}$ and fix a lift $\bA$ of $A$, this means a matrix $\bA \in \KK^{d \times n}$ with $\sval(\bA) = A$.
  For a vector $\blambda \in \KK_{\geq 0}$ with $\sum_{j = 1}^{n} \blambda_i = 1$ the valuation $x = \sval(\blambda)$ is in $\Tgeq^n$ and fulfills $\bigoplus_{j = 1}^{n} x_i = 0$.
  We want to show that $b = \sval(\lambda_1 \cdot \mathbf{a^{(1)}} + \dots + \lambda_n \cdot \mathbf{a^{(n)}}) \in \tconv{A}$.
  For each $i \in [d]$, let $c_i = \max\SetOf{|\sval(\lambda_j \cdot \mathbf{a^{(j)}_i})|}{j \in [n]}$. %, and let $J_i \subseteq [n]$ be the index set, at which the maximum is attained.
  Furthermore, we define $p = A \odot x = \bigoplus_{j \in [n]} x_j \odot a^{(j)} = \bigoplus_{j \in [n]} \sval(\lambda_j \cdot \mathbf{a^{(j)}})$.
  %% \[
  %% d = A \odot x = (x_1 \odot a^{(1)}, \ldots, x_n \odot a^{(n)}) \odot 0 = (\sval(\lambda_1 \cdot \mathbf{a^{(1)}}),\ldots, \sval(\lambda_n \cdot \mathbf{a^{(n)}}))) \odot 0 \enspace .
  %% \]
  We fix an $i \in [d]$ and we want to show that $b_i \in \Uncomp(p_i)$.
  Note that $|p_i| = c_i$. 
  If $p_i$ is not balanced, we already have $b_i = p_i$.
  Otherwise, we get $|b_i| \leq c_i$ and consequently $b_i \in [\ominus c_i, c_i]$. This finishes the proof of the inclusion '$\supseteq$'.

  \smallskip

  For the other direction, we fix $b \in \Uncomp(A \odot x)$ for some $x \in \Tgeq^n, \bigoplus_{j \in [n]} x_j = 0$.
  We define
  \[
  \lambda_j = t^{x_j} \cdot  \left(\sum_{k \in [n]}t^{x_k}\right)^{-1} \quad \text{ for each } j \in [n] \enspace .
  \]
  With this, we get $\lambda \geq 0$ and $\sum_{k \in [n]} \lambda_k = 1$.
  
  For each row $i \in [d]$, we denote by $J^+_i$ the set of indices of the positive elements in the set $\argmax\SetOf{a^{(i)}_j \odot x_j}{j \in [n]}$, and by $J^-_i$ analogously for negative elements.

  %We distinguish two cases.
  %If $|J^+_i| \neq |J^-_i|$, we set $\ell_i$ to be an arbitrary index in $\argmin\{|J^+_i|,|J^-_i|\}$.
  We set $\ell_i$ to be an arbitrary index in $\argmin\{|J^+_i|,|J^-_i|\}$, and define
  \[
  \mathbf{a_i^{(j)}} = 
  \begin{cases}
    \tsgn(a_{ij})t^{|a_{ij}|} + \alpha_i & \text{ for } j = \ell_i \\
    \tsgn(a_{ij})t^{|a_{ij}|} &  \text{ else } \enspace , 
  \end{cases}
  \]
  where
  \[
  \alpha_i = \dfrac{1}{\lambda_{\ell_i}}\left(-\sum_{k \in [d]} \tsgn(a_{ik})t^{|a_{ik}|+x_k} + \tsgn(b_i)t^{|b_i|}\right) \enspace .
  \]
  Note that $|b_i| \leq |a_{i\ell_i}| + x_{\ell_i}$ and $|a_{ij}| + x_j - x_{\ell_i} \leq |a_{i\ell_i}|$ for all $j \in [n]$. 
  Therefore, $\sval(\mathbf{a_i^{(j)}}) = a_{ij}$ for all $i \in [d]$ and $j \in [n]$.
  Furthermore, we get
  \begin{equation*}
    \begin{aligned}
      \mathbf{a_i} \cdot \lambda &= \sum_{k \in [d] \setminus \{\ell_i\}} \lambda_k\mathbf{a_i^{(k)}} + \lambda_{\ell_i}\mathbf{a_i^{(\ell_i)}}\\  &= \sum_{k \in [d]} \tsgn(a_{ik})t^{|a_{ik}|+x_k} -\sum_{k \in [d]} \tsgn(a_{ik})t^{|a_{ik}|+x_k} + \tsgn(b_i)t^{|b_i|} \\
      &= \tsgn(b_i)t^{|b_i|} \enspace .
    \end{aligned}
  \end{equation*}
  Hence, we have $\sval(\mathbf{a_i} \cdot \lambda) = b_i$ for all $i \in [d]$.
  This concludes the proof. 
\end{proof}
}
\iffull\prooflifthull\fi

\begin{remark}
  Theorem~\ref{thm:lift+hull} generalizes~\cite[Proposition 2.1]{DevelinYu:2007}, since $\val$ is a semiring homomorphism from $\KK_{\geq 0}$ to $\TTmax = \Tgeq$.
\end{remark}

\begin{corollary}
  The tropical convex hull is the union of the convex hulls of the lifts, i.e., 
  \[
  \tconv(A) = \sval( \conv( \sval^{-1} (A) )) \enspace .
  \]
\end{corollary}

%The latter maybe also follows from~\cite[Theorem 6.5]{JellScheidererYu:2018}. However, they range over all inequalities fulfilled by the lift. We only want linear inequalities or inequalities related to generators. 

\subsection{Convex hull from hyperoperation} \label{subsec:hull+hyperoperation}
As we shall see in Proposition~\ref{prop:image+hyperoperation}, our signed tropical convex hull from Definition~\ref{def:inner+hull} is a natural generalization of the classical convex hull as image of a simplex to hyperoperations.
In recent years, hyperfields found their way into matroid theory due to the work~\cite{BakerBowler:2016} building on~\cite{Krasner:1983} and~\cite{Viro:2010}.
While hyperfields have been used in tropical
geometry~\cite{JellScheidererYu:2018} from an algebraic point of view,
they were not used to describe intrinsically defined geometric objects
before. 

We introduce the necessary notions for hyperfields to define a signed convex hull and compare our binary operations with hyperfield operations~\eqref{eq:semialgebraic+order+hyperfield}.
Let us briefly introduce the \emph{real plus-tropical hyperfield} $\STH$, see~\cite{Viro:2010}.
It has the multiplicative group $(\TTpm, \odot)$ and its additive hyperoperation on $\TTpm$ is given by
\begin{equation} \label{hyperfield+addition}
  x \boxplus y = \begin{cases}
    \argmax_{x,y}(|x|,|y|) & \text{ if } \chi \subseteq \{+,\Zero\} \text{ or } \chi = \{-\} \\
    [\ominus |x|, |x|] & \text{ else } \enspace .
  \end{cases}
\end{equation}
We see that the latter addition for non-balanced numbers $x,y \in \TTpm$ differs from the Definition in~\eqref{eq:balanced+addition} in that it has a multi-valued result in the powerset of $\TTpm$. A balanced outcome $z \in \Tzero$ is replaced with the interval $\Uncomp(z) = [-|z|,|z|]$.
One can extend the operations again componentwise and use the symbol $\boxdot$ for the product of two matrices or vectors.
In particular, the operation $\boxdot$ agrees with $\odot$ on $\TTpm$. 
The addition is set-valued in $\STH$ if and only if it would be balanced in $\TSS$.
It agrees with $\oplus$ on $\Tgeq$. 

We recall the order relations used in~\cite{JellScheidererYu:2018} for the multiplicative real tropical hyperfield. Note that they use the multiplication $\odot = $`$\cdot$' instead of our approach with `$+$'. 

A polynomial over the real tropical hyperfield is a formal expression
\[
F(x) = \boxplus_{d_1,\ldots,d_n \in \ZZ} c_{d_1,\ldots,d_n} x_1^{d_1}\cdots x_n^{d_n}
\]
which can be evaluated at an element $\zeta \in \STH^n$. This yields a subset
\begin{equation} \label{eq:multivalued+evaluation}
F(\zeta) = \boxplus_{d_1,\ldots,d_n \in \ZZ} c_{d_1,\ldots,d_n} \zeta_1^{d_1}\cdots \zeta_n^{d_n} \subseteq \STH^n \enspace .
\end{equation}
Note that we mainly deal with linear polynomials, where the exponent vector $(d_1,\dots,d_n) \in \ZZ^d$ is just a unit vector. 

One can define the sets
\begin{equation} \label{eq:semialgebraic+order+hyperfield}
  \begin{aligned}
    \{F = 0\} &:= \{(x_1,\ldots,x_n) \in \TTpm^n \colon \Zero \in F(x_1,\ldots,x_n) \} \\
    \{F \geq 0\} &:= \{(x_1,\ldots,x_n) \in \TTpm^n \colon F(x_1,\ldots,x_n) \cap \Tgeq \neq \emptyset \} \\
    \{F > 0\} &:= \{(x_1,\ldots,x_n) \in \TTpm^n \colon F(x_1,\ldots,x_n) \in \Tgt \} \enspace . 
  \end{aligned}
\end{equation}
Observe that $F = 0$ is indeed equivalent to $F \geq 0 \vee -F \leq 0$ due to the structure of the set~\eqref{eq:multivalued+evaluation}.
Translating~\eqref{eq:semialgebraic+order+hyperfield} to the symmetrized semiring $\TSS$ yields the relations `$\beq$', `$\teq$' and `$>$'.

To motivate the next construction, we consider a tropical polytope generated by $V \in \Tgeq^{d \times k}$ as the image of the tropical standard simplex in the sense that
\[
\tconv(V) = \{V \odot \lambda \mid \bigoplus_{\ell \in [k]} \lambda_{\ell} = 0 \} \enspace .
\]

For a matrix $A \in \TTpm^{d \times n}$, we define the \emph{balanced image} of the tropical standard simplex $\Delta_d = \{\lambda \mid \bigoplus_{\ell \in [k]} \lambda_{\ell} = 0 \}$ by
\begin{equation} 
 A \odot \Delta_d := \SetOf{A \odot x}{\bigoplus_{j \in [n]} x_j = 0, x \geq \Zero } \subset \TSS^d \enspace .
\end{equation}
With this notion, one can write $\tconv(A) = \bigcup_{z \in A \odot \Delta_d} \Uncomp(z)$.

By using the hyperoperations in $\STH$, we can naturally consider the image of the tropical standard simplex $\Delta_d = \{\lambda \mid \bigoplus_{\ell \in [k]} \lambda_{\ell} = 0 \}$ with respect to matrix multiplication by $V \in \TTpm^{d \times k}$ as a subset of $\TTpm^d$.

\begin{proposition} \label{prop:image+hyperoperation}
\[
V \boxdot \Delta_d = \tconv(V) \enspace .
\]
\end{proposition}
\begin{proof}
  This follows directly by the definition of the set-valued addition in~\eqref{hyperfield+addition} from~\eqref{eq:hull+union+intervals} with $\Uncomp(z) = [\ominus |z|, |z|]$ for $z \in \Tzero$. 
\end{proof}

\newcommand{\Bconv}{
Parallel to the development of tropical convexity, the more general notion of $\BB$-convexity was developed starting with~\cite{BriecHorvath:2004}.
The notion of $\BB$-convexity boils down to convexity defined over the semiring $\RR_{\geq 0}$ with operations `$\oplus$' $=$ `$\max$' and `$\odot$' $=$ `$\cdot$', see~\cite[Theorem 2.1.1]{BriecHorvath:2004}.
Taking logarithms transforms these operations to `$\oplus$' $=$ `$\max$' and `$\odot$' $=$ `$+$' on $\RR \cup \{-\infty\}$.
This gives rise to a transferred version of $\BB$-convexity on $\TTpm$ by considering the images of $\BB$-convex sets in $\RR^d$ under the map $\slog \colon x \mapsto \sgn(x) \log(|x|)$.

 The following example shows that our notion of signed tropical convexity is an even more restrictive notion than $\BB$-convexity and $\BB^{\sharp}$-convexity~\cite{Briec:2015}.
\begin{example} \label{ex:different+convexities}
  The tropical convex hull of $A = \{(\ominus 2, \ominus 1), (2,1)\}$ is the set
  \[
  [\ominus 2, 2] \times [\ominus 1,1] \enspace . %\cup 
  \]
  However, the set $\Co^r(A)$ is
  \[
  L = \SetOf{(2\odot\lambda,\lambda)}{\lambda \in [\ominus 1,1]} \enspace .
  \]
  for all $r \in \NN$.
  In particular, also $\Co^{\infty}(A)$ equals $L$.
  This implies that $\BB(L) = L$.
  We depict both in Figure~\ref{fig:BB+convex+trop+convex}.
   Hence, $\tconv(A)$ strictly contains $\BB(A)$. 
   Furthermore,~\cite[Corollary 4.2.4]{Briec:2015} shows that $L$ is also $\BB^{\sharp}$-convex.

   Interestingly, the set $L$ is also the image under the signed valuation of the set
  \[
  \conv\left( \pvec{-t^2}{-t}, \pvec{t^2}{t} \right) \enspace .
  \]
  Here, we mean the convex hull over the Puiseux series $\pseries{R}{t}$.
  So $L$ is the tropicalization of a single line segment while our hull construction yields the union of line segments whose spanning points tropicalize to $A$, as we saw in Section~\ref{subsec:puiseux+lifts}.
  For example, we get the set $\{\ominus 2\} \times [\ominus 1, 1] \cup [\ominus 2, 2] \times \{1\}$ as the tropicalization of $\conv\left(\pvec{-2t^2}{-t},\pvec{t^2}{2t}\right)$.
\end{example}

\begin{figure}[ht]
  \begin{tikzpicture}

  \Koordinatenkreuz{-3.2}{3.3}{-2.2}{2.3}{$x_1$}{$x_2$}
  \Gitter{-3.1}{3.1}{-2.1}{2.1}

%  \coordinate (center) at (1.5,1.5){};
  \coordinate (pp) at (2,1){};
  \coordinate (mp) at (-2,1){};
  \coordinate (pm) at (2,-1){};
  \coordinate (mm) at (-2,-1){};
  
%  \draw[Linesegment] (center) -- (pp);
  \draw[fill,Linesegment,opacity=0.4] (pp) -- (mp) -- (mm) -- (pm) -- cycle;

  \node[Endpoint] at (pp){};
  \node[Endpoint] at (mm){};

  \draw[purple, very thick] (pp) -- (mm);
  
\end{tikzpicture}
  \caption{Distinction between $\BB$-convex line and tropical line segment through the origin. }
  \label{fig:BB+convex+trop+convex}
\end{figure}

 \begin{remark} \label{rem:cancellative+sum}
It is tempting to define a \emph{cancellative sum} for two numbers $a, b \in \TTpm$ by
\[
a \overline{\oplus} b =
\begin{cases}
  a & |a| > |b| \\
  b & |b| > |a| \\
  a & a = b \\
  \Zero & a = \ominus b
\end{cases}
\enspace .
\]
This can be extended componentwise to $\TTpm^d$.

An iterative version of this construction is used in~\cite{Briec:2015}.
A conceptional drawback of the cancellative sum is that it is not associative, as the example
  \[
  0 \cancsum (\ominus 0 \cancsum -1) = 0 \cancsum \ominus 0 = \Zero \neq -1 = \Zero \cancsum -1 = ( 0 \cancsum \ominus 0) \cancsum -1 % \enspace ,
  \]
  shows.
%  We use a similar but multi-valued version in Section~\ref{subsec:fourier-motzkin} for~\eqref{eq:multi-valued+cancellation}.
\end{remark}

   \begin{remark}
     Recall that Theorem~\ref{thm:generators+for+all+orthants} gave a way to determine the whole tropical convex hull from the tropical convex hull in each orthant.
     For sufficiently generic matrices, one can use the cancellative sum from Remark~\ref{rem:cancellative+sum}.
  If there are no antipodal points, no balanced coefficients arise in the elimination and the map $\xi$ used for Theorem~\ref{thm:replace+balanced} is just the identity. 
  Hence, the iterative construction of a single intersection point with a coordinate hyperplane suffices.      
   \end{remark}

}

\iffull
\subsection{Connection with $\BB$-convexity} \label{subsec:limit+hull}

Parallel to the development of tropical convexity, the more general notion of $\BB$-convexity was developed starting with~\cite{BriecHorvath:2004}.
The notion of $\BB$-convexity boils down to convexity defined over the semiring $\RR_{\geq 0}$ with operations `$\oplus$' $=$ `$\max$' and `$\odot$' $=$ `$\cdot$', see~\cite[Theorem 2.1.1]{BriecHorvath:2004}.
Taking logarithms transforms these operations to `$\oplus$' $=$ `$\max$' and `$\odot$' $=$ `$+$' on $\RR \cup \{-\infty\}$.
This gives rise to a transferred version of $\BB$-convexity on $\TTpm$ by considering the images of $\BB$-convex sets in $\RR^d$ under the map $\slog \colon x \mapsto \sgn(x) \log(|x|)$.

 The following example shows that our notion of signed tropical convexity is an even more restrictive notion than $\BB$-convexity and $\BB^{\sharp}$-convexity~\cite{Briec:2015}.
\begin{example} \label{ex:different+convexities}
  The tropical convex hull of $A = \{(\ominus 2, \ominus 1), (2,1)\}$ is the set
  \[
  [\ominus 2, 2] \times [\ominus 1,1] \enspace . %\cup 
  \]
  However, the set $\Co^r(A)$ is
  \[
  L = \SetOf{(2\odot\lambda,\lambda)}{\lambda \in [\ominus 1,1]} \enspace .
  \]
  for all $r \in \NN$.
  In particular, also $\Co^{\infty}(A)$ equals $L$.
  This implies that $\BB(L) = L$.
  We depict both in Figure~\ref{fig:BB+convex+trop+convex}.
   Hence, $\tconv(A)$ strictly contains $\BB(A)$. 
   Furthermore,~\cite[Corollary 4.2.4]{Briec:2015} shows that $L$ is also $\BB^{\sharp}$-convex.

   Interestingly, the set $L$ is also the image under the signed valuation of the set
  \[
  \conv\left( \pvec{-t^2}{-t}, \pvec{t^2}{t} \right) \enspace .
  \]
  Here, we mean the convex hull over the Puiseux series $\pseries{R}{t}$.
  So $L$ is the tropicalization of a single line segment while our hull construction yields the union of line segments whose spanning points tropicalize to $A$, as we saw in Section~\ref{subsec:puiseux+lifts}.
  For example, we get the set $\{\ominus 2\} \times [\ominus 1, 1] \cup [\ominus 2, 2] \times \{1\}$ as the tropicalization of $\conv\left(\pvec{-2t^2}{-t},\pvec{t^2}{2t}\right)$.
\end{example}

\begin{figure}[ht]
  \begin{tikzpicture}

  \Koordinatenkreuz{-3.2}{3.3}{-2.2}{2.3}{$x_1$}{$x_2$}
  \Gitter{-3.1}{3.1}{-2.1}{2.1}

%  \coordinate (center) at (1.5,1.5){};
  \coordinate (pp) at (2,1){};
  \coordinate (mp) at (-2,1){};
  \coordinate (pm) at (2,-1){};
  \coordinate (mm) at (-2,-1){};
  
%  \draw[Linesegment] (center) -- (pp);
  \draw[fill,Linesegment,opacity=0.4] (pp) -- (mp) -- (mm) -- (pm) -- cycle;

  \node[Endpoint] at (pp){};
  \node[Endpoint] at (mm){};

  \draw[purple, very thick] (pp) -- (mm);
  
\end{tikzpicture}
  \caption{Distinction between $\BB$-convex line and tropical line segment through the origin. }
  \label{fig:BB+convex+trop+convex}
\end{figure}

 \begin{remark} \label{rem:cancellative+sum}
It is tempting to define a \emph{cancellative sum} for two numbers $a, b \in \TTpm$ by
\[
a \overline{\oplus} b =
\begin{cases}
  a & |a| > |b| \\
  b & |b| > |a| \\
  a & a = b \\
  \Zero & a = \ominus b
\end{cases}
\enspace .
\]
This can be extended componentwise to $\TTpm^d$.

An iterative version of this construction is used in~\cite{Briec:2015}.
A conceptional drawback of the cancellative sum is that it is not associative, as the example
  \[
  0 \cancsum (\ominus 0 \cancsum -1) = 0 \cancsum \ominus 0 = \Zero \neq -1 = \Zero \cancsum -1 = ( 0 \cancsum \ominus 0) \cancsum -1 % \enspace ,
  \]
  shows.
%  We use a similar but multi-valued version in Section~\ref{subsec:fourier-motzkin} for~\eqref{eq:multi-valued+cancellation}.
\end{remark}

   \begin{remark}
     Recall that Theorem~\ref{thm:generators+for+all+orthants} gave a way to determine the whole tropical convex hull from the tropical convex hull in each orthant.
     For sufficiently generic matrices, one can use the cancellative sum from Remark~\ref{rem:cancellative+sum}.
  If there are no antipodal points, no balanced coefficients arise in the elimination and the map $\xi$ used for Theorem~\ref{thm:replace+balanced} is just the identity. 
  Hence, the iterative construction of a single intersection point with a coordinate hyperplane suffices.      
   \end{remark}

\else
In Appendix~\ref{app:b-conv}, we compare our convexity notion to
$\BB$-convexity studied in the literature ~\cite{BriecHorvath:2004},
as well as a cancellative sum construction used in \cite{Briec:2015}.
\fi

\section{Farkas' Lemma and Fourier-Motzkin elimination} \label{sec:Farkas+FM}

\subsection{Convexity and tropical linear feasibility} \label{subsec:dual+notions}
For a matrix $A =(a_{ij})\in \TTpm^{d \times n}$, we define the \emph{non-negative kernel}
\begin{equation} \label{eq:non+negative+kernel}
  \begin{aligned}
    \nnker(A) &= \SetOf{x \in \Tgeq^n \setminus \{\Zero\}}{A \odot x
      \beq \Zero}
   % &= \SetOf{x \in \TTpm^n}{A \odot x \beq \Zero, x \teq \Zero, x \neq \Zero} \enspace .
  \end{aligned}
\end{equation}
This corresponds to the classical definition of a polyhedral cone in
the form $Ax = 0, x \geq 0, x\neq 0$. We replace `$=$' by `$\beq$' and
`$\geq$' by `$\teq$'.
In terms of the non-negative kernel, we can express containment in the
convex hull as follows.

\begin{restatable}{proposition}{eqhullsol} \label{prop:equivalence+hull+solution}
  For $A \in \TTpm^{d \times n}$ and $b \in \TTpm^d$ we have
  \[
  b \in \tconv(A) \Leftrightarrow
  \nnker
  \begin{pmatrix}
    A & \ominus b \\
    0 & \ominus 0
  \end{pmatrix}
  \neq \emptyset \enspace .
  \]
\end{restatable}
\newcommand{\proofeqhullsol}{
\begin{proof}
  The condition $b \in \tconv(A)$ is equivalent to the existence of an
  element $x \in \Tgeq^n$ with $\bigoplus_{j \in [n]} x_j = 0$ and
  $A\odot x\beq b$.

Let $(x,t)$ be a vector in the non-negative kernel, where $x\in
\Tgeq^n$ and $t\in \Tgeq$ denotes the last component. First, we claim
that $t\neq\Zero$. Indeed, $t=\Zero$ would yield $\bigoplus_{j \in
  [n]} x_j \beq \Zero$, which implies $x_j=\Zero$ for all $j\in [n]$. Thus, we
obtain $(x,t)=\Zero$, a contradiction. Since $t\neq\Zero$,  we can
scale $(x,t)$ such that $t=0$. In this case, the definition of the
kernel gives $A\odot x\ominus b\beq \Zero$ and $\bigoplus_{j \in [n]}
x_j\ominus 0 \beq \Zero$. The latter inequality yields $\bigoplus_{j
  \in [n]} x_j=0$.
This is the same as the combination witnessing $b \in \tconv(A)$ as
above.
\end{proof}}
\iffull\proofeqhullsol\fi

\begin{corollary} \label{cor:zero+in+convex+hull}
  The origin $\Zero$ is in the convex hull $\tconv(A)$ if and only if the non-negative kernel $\nnker(A)$ is not empty. 
\end{corollary}
\begin{proof}
  Setting $b = \Zero$ in Proposition~\ref{prop:equivalence+hull+solution} implies the equivalence with the definition from~\eqref{eq:non+negative+kernel}. 
\end{proof}

\iffalse
\begin{corollary} \label{cor:equation+form+inner+hull}
  The convex hull $\tconv(A)$ is given by
  \[
  \SetOf{z \in \TTpm^d}{  \nnker
  \begin{pmatrix}
    A & \ominus z \\
    0 & \ominus 0
  \end{pmatrix}
  \neq \emptyset} \enspace .
  \]
  or, equivalently,
  \begin{equation} \label{eq:balanced+equation+version+inner+hull}
  \SetOf{z \in \TTpm^d}{ \exists x \in \Tgeq^n \colon A \odot x \beq z,\; \bigoplus_{j \in [n]}x = 0} \enspace .
  \end{equation}
\end{corollary}
\fi
%% \begin{example}
%%   \todo[inline]{GL: Small example with Figure. }
%% \end{example}

We now define the \emph{open tropical cone} as the dual to the non-negative kernel \eqref{eq:non+negative+kernel}.
\begin{equation} \label{eq:def+sep}
  \sep(A) = \{y \in \TTpm^d \mid \trans{y} \odot A > \Zero \} \enspace .
\end{equation}
The name is motivated by the use of the elements of $\sep(A)$ as separators of the columns of $A$ from the origin. 
Note that the condition `$ > \Zero$', in particular, means that the product `$\trans{y} \odot A$' is comparable with $\Zero$ and, equivalently, is in $\Tgt$.

\smallskip

We can also define $\nnker(A)$ and $\sep(A)$ for $A \in \TSS^{d \times n}$.
However, this does not provide a wider class of objects.
This follows by replacing a balanced number by $\Zero$ in $\nnker(A)$ and applying Theorem~\ref{thm:replace+balanced} for $\sep(A)$.
We still extend the definition to these more general matrices, as it will lead to simplified arguments. 

\begin{example}
In the next example we work with the semiring
$(\RR_{\geq 0},\max,\cdot)$, as
it provides a more natural picture of the behaviour around the origin
in the different orthants.
%; cf. Figure~\ref{fig:first+hull} for the visualization of $(\max,+)$. 
  Instead of the semiring defined on $\RR \cup \{-\infty\}$ with operations `$\max$' and `$+$', one can isomorphically use the semiring $\RR_{> 0} \cup \{0\}$ with `$\max$' and `$\cdot$'.
  The exponentiation map
  \[
  \exp \colon \RR \cup \{-\infty\} \to \RR_{> 0} \cup \{0\}
  \]
  yields an isomorphism between these two semirings.
  One can further extend this mapping to $\TTpm$ by setting
  \begin{equation*}
    \begin{aligned}
      x \mapsto \qquad \tsgn(x)\exp(|x|) \enspace .
    \end{aligned}
  \end{equation*}
  Note that the image of $\TSS$ involves a balanced version of positive numbers as it decomposes into the union $\RR_{> 0} \cup \RR_{<0} \cup \bullet\RR_{>0} \cup \{0\}$.
  
  The dotted black shape in Figure~\ref{fig:separators+plane} depicts the $(\max,\cdot)$-convex hull of the columns of the matrix $A = \begin{pmatrix} 3 & 1 & -2 \\ 1 & 4 & 3 \end{pmatrix}$.
  The combinations, where balanced numbers occur, are
  \[
  \frac{1}{2} \cdot \pvec{-2}{3} \oplus \pvec{1}{4} = \pvec{\bullet 1}{4} \quad \text{ and } \quad \frac{2}{3} \cdot \pvec{3}{1} \oplus \pvec{-2}{3} = \pvec{\bullet 2}{3} \enspace .
  \]
  The blue shaded area is $\sep(A)$.
  The configuration visualizes Theorem~\ref{thm:weak+duality}, as $\sep(A)$ is not empty and the $(\max,\cdot)$-convex hull of $A$ does not contain the origin. 
\end{example}

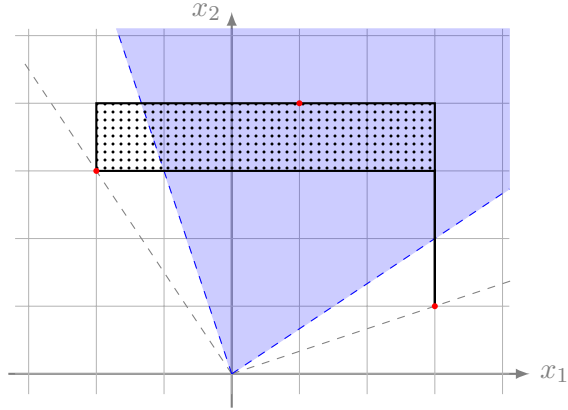
\begin{figure}
  \begin{tikzpicture}[scale = 0.9]

  \newcommand\mmeps{0.1}

  \Koordinatenkreuz{-3.3}{4.4}{-0.5}{5.35}{$x_1$}{$x_2$}
  \Gitter{-3.2}{4.1}{-0.3}{5.1}

  \coordinate (A) at (3,1){};
  \coordinate (B) at (1,4){};
  \coordinate (C) at (-2,3){};

  \coordinate (canc1) at (0,4){};
  \coordinate (canc2) at (0,3){};
  \coordinate (break1) at (3,4){};
  \coordinate (break2) at (-2,4){};
  
  \coordinate (H1) at (4.2,4.2/3){};
  \coordinate (H2) at (-3.1,3.1*3/2){};

  \coordinate (O1) at (-1.7,1.7*3){};
  \coordinate (O2) at (4.1,4.1*2/3){};

  \coordinate (origin) at (0,0) {};
  
  \draw[dashed, gray] (origin) -- (H1);
  \draw[dashed, gray] (origin) -- (H2);

  \draw[dashed, blue] (origin) -- (O1);
  \draw[dashed, blue] (origin) -- (O2);

  \draw[fill,blue,opacity=0.2] (origin) -- (O1) -- (4.1,5.1) -- (O2) -- cycle;

  \draw[thick,pattern=dots] (A) -- (break1) -- (break2) -- (C) -- (3,3) -- cycle;

  \node[Endpoint] at (A){};
  \node[Endpoint] at (B){};
  \node[Endpoint] at (C){};
   
\end{tikzpicture}
  \caption{An open tropical cone visualized with the operations `$\oplus$' $=$ `$\max$' and `$\odot$' $=$ `$\cdot$'. }
  \label{fig:separators+plane}
\end{figure}

Let us now show that weak duality holds between the non-negative kernel and the open tropical cone:
\begin{theorem}[Weak duality] \label{thm:weak+duality}
  For a matrix $A \in \TTpm^{d \times n}$ at least one of the sets $\nnker(A)$ and $\sep(A)$ is empty. 
  Equivalently, for an arbitrary vector $y \in \TTpm^d$ and a non-negative vector $x \in \Tgeq^n \setminus \{\Zero\}$ at most one of $A \odot x \in \TT_{\bullet}^d$ and $\trans{y} \odot A > \Zero$ can hold.
\end{theorem}
\begin{proof}
  The claim follows from the fact that $\TT_{\bullet}$ is a left- and right-ideal of $\TSS$, see~\cite[Definition 2.6]{AkianGaubertGuterman:2014}.
  This means that for $x \in \nnker(A)$, the product $\trans{y} \odot A \odot x$ is in $\TT_{\bullet}$ while $\trans{y} \odot A > \Zero \Rightarrow \trans{y} \odot A \odot x > \Zero$. 
\end{proof}

\begin{remark}
  We give a direct proof of the former statement. 
  We consider the product $\trans{y} \odot A \odot x$.
  Scaling the rows of $A$ by arbitrary numbers in $\TTpm$ does not change whether $x \in \nnker(A)$, just as scaling the columns of $A$ by a non-negative number in $\Tgeq$ does not change whether $y \in \sep(A)$.
  Hence, we can assume that $x$ and $y$ only have the entries $0$ or $\Zero$.
  Let $a_{pq}$ be the entry of $A$ with maximal absolute value.
  For $x \in \nnker(A)$, there is an index $r \in [n]$ such that $a_{pr} = \ominus a_{pq}$.
  We can assume that $a_{pq} > \Zero > a_{pr}$.
  From this we conclude that $y \not\in \sep(A)$ since the $r$th column of $\trans{y} \odot A$ cannot be positive.
\end{remark}

The key result of this section will be showing the appropriate version
of Farkas' lemma. The proof will follow via Fourier-Motzkin elimination. 
\begin{theorem}[Farkas' lemma] \label{thm:farkas+lemma}
  For a matrix $A \in \TTpm^{d \times n}$ exactly one of the sets $\nnker(A)$ and $\sep(A)$ is nonempty. 
\end{theorem}

\begin{remark} \label{rem:connection+nonnegative+Farkas}
  Theorem~\ref{thm:farkas+lemma} is similar to~\cite[Corollary~3.12]{GrigorievPodolskii:2018}.
  Through a suitable replacement of the balanced coefficients and a careful analysis of the occuring signs, Theorem~\ref{thm:farkas+lemma} may be deduced from~\cite{GrigorievPodolskii:2018}.
  Note however, that we allow for unconstrained variables in the definition of $\sep(A)$ which is not directly covered by~\cite[Corollary~3.12]{GrigorievPodolskii:2018}.  
\end{remark}

\subsection{Technical properties of the order relations}

The next lemma is a version of transitivity and it is a preparation for the elimination of a variable in a system of inequalities in Section~\ref{subsec:fourier-motzkin}. 
\iffull\else The proofs of the next two propositions are given in the Appendix.\fi
\begin{restatable}{proposition}{elimstrict} \label{prop:elimination+system}
  Let $A, B \subset \TSS$ be two finite sets. 
  There is an element $c \in \TSS$ with
  \begin{equation} \label{eq:pairs+with+element}
    c \ominus a > \Zero \text{ and } b \ominus c > \Zero \text{ for all } a \in A, b \in B
  \end{equation}
  if and only if
  \begin{equation} \label{eq:pairs+greater+zero}
    b \ominus a > \Zero \text{ for all } (a,b) \in A \times B \enspace .
  \end{equation}
  Furthermore, the element $c$ can chosen to be signed. 
\end{restatable}
\newcommand{\proofelimstrict}{
\begin{proof}
  For each pair $(a,b) \in A \times B$, we add the inequalities in~\eqref{eq:pairs+with+element} using Lemma~\ref{lem:adding+strict+inequalities} and obtain $b \ominus a \oplus c \ominus c > \Zero$.
  This implies $b \ominus a > |c| \geq \Zero$ and, hence,~\eqref{eq:pairs+greater+zero}.
  
  For the other direction, note that $A \subset \TTpm$ or $B \subset \TTpm$, as two balanced elements are not comparable by '$>$'.
  Because of the symmetry~\eqref{eq:symmetry+reflection+order}, we can assume that $B \subset \TTpm$.
  Let $\beta$ denote the minimum of $B$.
  Furthermore, we define $\alpha$ as the maximum of $A \cap \TTpm$ and $\{|a| \mid a \in A\cap \TT_{\bullet}\}$, where either of these two sets could also be empty.
  We obtain from~\eqref{eq:pairs+greater+zero} that $\beta > \alpha$, where we use that $b > a \Leftrightarrow b > |a|$ for $a \in \Tzero$.
  An arbitrary element $c$  with $\alpha<c<\beta$
  fulfills~\eqref{eq:pairs+with+element}.
  As the elements in $B$ are totally ordered, the claim for the inequalities involving $B$ follows immediately.
  Distinguishing the balanced and signed elements yields the claim for the inequalities involving $A$. 
\end{proof}
}
\iffull\proofelimstrict\fi
\begin{corollary} \label{cor:non-strict-relation-transitive}
  The non-strict relation $x\ge y$ defined in~\eqref{eq:extended+partial+order+non+strict} is a partial order. 
\end{corollary}
\begin{proof}
  Reflexivity and antisymmetry follow directly from~\eqref{eq:extended+partial+order+non+strict} and~\eqref{eq:extended+partial+order}.
  Proposition~\ref{prop:elimination+system} implies transitivity. 
\end{proof}

\begin{restatable}{proposition}{elimnonstrict} \label{prop:elimination+system+non+strict}
  Let $A, B \subset \TSS$ be two finite sets. 
  There is an element $c \in \TTpm$ with
  \begin{equation} \label{eq:pairs+with+element+non+strict}
    c \ominus a \teq \Zero \text{ and } b \ominus c \teq \Zero \text{ for all } a \in A, b \in B
  \end{equation}
  if and only if
  \begin{equation} \label{eq:pairs+greater+zero+non+strict}
    b \ominus a \teq \Zero \text{ for all } (a,b) \in A \times B \enspace .
  \end{equation}
\end{restatable}
\newcommand{\proofelimnonstrict}{
\begin{proof}
  
  The first direction from~\eqref{eq:pairs+with+element+non+strict} to~\eqref{eq:pairs+greater+zero+non+strict} follows from Lemma~\ref{lem:basic+properties+teq}\eqref{item:teq+other+side} and~\ref{lem:basic+properties+teq}\eqref{item:restricted+signed+transitivity} because of $c \in \TTpm$.  

  \smallskip

  For the other direction, let
  \begin{equation} \label{eq:possible+bounds+balanced+signed}
    \begin{aligned}
      \alpha_0 &= \argmin\SetOf{|a|}{a \in A \cap \Tzero} \enspace , \\
      \alpha_1 &= \max\SetOf{a}{a \in A \cap \TTpm} \enspace , \\
      \beta_0 &= \argmin\SetOf{|b|}{b \in B \cap \Tzero} \enspace , \\
      \beta_1 &= \min\SetOf{b}{b \in B \cap \TTpm} \enspace ,
    \end{aligned}
  \end{equation}
  with respect to the ordering '$\geq$'. 
  By construction, we get from~\eqref{eq:pairs+greater+zero+non+strict} the relation $\beta_1 \ominus \alpha_1 \teq \Zero$, which yields $\beta_1 \geq \alpha_1$.
  Furthermore, we obtain $\beta_1 \ominus \alpha_0 \teq \Zero$ implying $\beta_1 \geq \ominus |\alpha_0|$ and $\beta_0 \ominus \alpha_1 \teq \Zero$ implying $|\beta_0| \geq \alpha_1$.
  We conclude that
    \[
  \underline{\beta} := \min(|\beta_0|,\beta_1) \geq \max(\ominus|\alpha_0|,\alpha_1) =: \overline{\alpha} \enspace ,
  \]
  using also the trivial inequality $|\beta_0| \geq \ominus |\alpha_0|$.
  Let $\gamma$ be an arbitrary element in the interval
  \[
  \SetOf{x \in \TTpm}{\underline{\beta} \geq x \geq \overline{\alpha}} \neq \emptyset .
  \]
  By checking all possibilities arising from the list in~\eqref{eq:possible+bounds+balanced+signed}, we see that the element $\gamma$ fulfills $b \ominus \gamma \teq \Zero$ and $\gamma \ominus a \teq \Zero$ for all $a \in A, b \in B$. 
\end{proof}
}

\iffull\proofelimnonstrict\fi

To deal with geometric objects in $\TTpm^d$, we will use balanced numbers because this allows for explicit calculations in the semiring $\TSS$.
However, as we are only interested in the signed part of the sets. 
We provide a first tool to resolve balanced numbers in inequalities.
While this is for strict inequalities, Proposition~\ref{prop:resolving+balanced+coefficiens+teq} provides a tool for the relation $\teq$. 

\begin{lemma} \label{lem:resolve+balanced}
  For $a,b \in \TSS$, we have an equivalence of
  \begin{enumerate}
  \item \label{eq:inequality+balanced} $a \ominus a \oplus b > \Zero $
  \item \label{eq:split+balanced} $a \oplus b > \Zero \text{ and } \ominus a \oplus b > \Zero $
  \item \label{eq:all+between+balanced} For all $c \in [\ominus |a|, |a|] = \Uncomp(a \ominus a)$, it holds $c \oplus b > \Zero $. 
  \end{enumerate}
\end{lemma}
\begin{proof}
  The condition~\eqref{eq:all+between+balanced} clearly implies~\eqref{eq:split+balanced}, as the latter is just a special case.
  The implication from~\eqref{eq:split+balanced} to~\eqref{eq:inequality+balanced} follows by adding up the positive values in~\eqref{eq:split+balanced}. For the direction from~\eqref{eq:inequality+balanced} to~\eqref{eq:all+between+balanced}, we use that $a \ominus a$ is balanced and, hence, incomparable. Therefore, we have $b > |a|$. Because of $|a| = |\ominus a| = |a \ominus a|$, the claim follows. 
\end{proof}

Given a matrix $A\in \TSS^{d\times n}$, let us construct the matrix
$B=\xi(A)\in \TTpm^{d'\times n}$ for $d\le d'\le 2d$ as follows. We
include in $B$  all the columns of $A$ that do not contain any balanced elements.
For every column $a_{* j}$ that contains some
balanced elements, we include two columns $a'$ and $a''$ in $B$, such
that  whenever $a_{ij}=\bullet \alpha$ is a balanced entry, then we
set $a'_{i}=\oplus \alpha$ and $a''_{i}=\ominus \alpha$. For all other
$i\in [d]$, we set $a'_i=a''_i=a_{ij}$.
\begin{theorem}\label{thm:replace+balanced}
For every matrix $A\in \TSS^{d\times n}$, $\sep(A)=\sep(\xi(A))$.
\end{theorem}
\begin{proof}
  Recalling Definition~\ref{eq:def+sep}, we see that it is enough to consider the transition from $A$ to $B$ column by column.
  An inequality arising from $A$ (after potentially reordering entries) is of the form $\trans{y} \odot (u,v) > \Zero$ with $u \in \TTpm^k$ and $v \in \Tzero^{d-k}$ for some $k \leq d$.
  It either remains the same in $B$ if it has no balanced entry or it gets transformed to two inequalities. 
  Regrouping the summands by balanced and signed numbers, we deduce the claim from the equivalence of~\eqref{eq:inequality+balanced} and~\eqref{eq:all+between+balanced} in Lemma~\ref{lem:resolve+balanced}. 
\end{proof}

\subsection{Fourier-Motzkin elimination}  \label{subsec:fourier-motzkin}

We derive three versions of Fourier-Motzkin elimination, which will be useful for deriving further description of signed tropically convex sets in Section~\ref{sec:halfspaces} and~\ref{sec:generators+in+orthants}.
As the elimination process produces balanced coefficients for the inequalities, it is convenient to have an elimination procedure which can directly deal with those (Theorem~\ref{thm:balanced+fourier+motzkin}).
We also need to derive explicit inequalities with signed coefficients (Corollary~\ref{cor:signed+fourier+motzkin}) to describe the dual convex hull in Section~\ref{sec:generators+in+orthants}.
The version with non-strict inequalities (Theorem~\ref{thm:non+strict+signed+fourier+motzkin}) will be used in constructing an exterior description by closed tropical halfspaces (Theorem~\ref{thm:signed+Minkowski+Weyl}).

\smallskip

For a subset $M$ of $\TTpm^d$, we define its coordinate projection $\rho_i(M)$ with $i \in [d]$ by
\begin{equation}
  \rho_i(M) = \{ (x_1,\ldots,x_{i-1},x_{i+1},\ldots,x_d) \mid (x_1,\ldots,x_d) \in M \} \enspace .
\end{equation}

We fix a matrix $A = (a_{ij}) \in \TSS^{d \times n}$.
The non-negative matrix $S^{(i)}$ in the group of signed tropical transformations from Remark~\ref{rem:transformation+group} with the sequence $(-|a_{ij}|)_j \in \Tgeq^n$ as diagonal scales the rows of $A$ such that the $i$th row of $A \odot S$ has only entries with absolute value $0$ or $\Zero$. 
As $S$ is non-negative, we directly obtain from the definitions in Section~\ref{subsec:dual+notions} the following.

\begin{lemma} \label{lem:transformation+normalized+row}
  There is a matrix $S^{(i)}\in \mathbb{D} \subset \TTpm^{d \times d}$ such that the entries in the $i$th row of $A \odot S^{(i)}$ have absolute value $0$ or $\Zero$, and
  $\nnker(A \odot S^{(i)}) = \nnker(A)$, $\sep(A \odot S^{(i)}) = \sep(A)$ and $\SetOf{y \in \TTpm^d}{\trans{y} \odot A \teq \Zero} = \SetOf{y \in \TTpm^d}{\trans{y} \odot S^{(i)}\odot A \teq \Zero}$.
\end{lemma}

We define several sets partitioning $[n]$ based on the sign of the entries in the $i$th row of $A$ by
\begin{equation} \label{eq:index+sets+signs}
  \begin{aligned}
    J^+ &= \{j \in [n] \mid a_{ij} > \Zero\}\,, &  J^- &= \{j \in [n] \mid a_{ij} < \Zero\} \\
    J^{\bullet} &= \{j \in [n] \mid a_{ij} \in \Tzero \setminus \{\Zero\}\}\,, & J^0 &= \{j \in [n] \mid a_{ij} = \Zero\} \enspace .
  \end{aligned}
\end{equation}

Furthermore, we define $T^{(i)} = (t_{j,p}) \in \{0,\Zero\}^{n \times ((J^+ \cup J^{\bullet}) \times (J^{\bullet} \cup J^-) \cup J^0)}$ as the incidence matrix
  \begin{equation} \label{eq:incidence+matrix+elimination}
  t_{j,p} =
  \begin{cases}
    0 & j = k \text{ or } j = \ell \text{ for } p = (k,\ell) \in (J^+ \cup J^{\bullet}) \times (J^{\bullet} \cup J^-)  \\
    0 & j = p \text{ for } p \in J^0 \\
    \Zero & \text{ else }
  \end{cases}
  \enspace .
  \end{equation}
  We denote by $A_{-i}$ the matrix obtained from $A$ by removing the last row. 

  \subsubsection{Strict inequalities}
  
\begin{theorem}[Fourier--Motzkin for strict inequalities with balanced numbers] \label{thm:balanced+fourier+motzkin}
  The $i$th coordinate projection of the open tropical cone $\rho_i(\sep(A))$ for the matrix $A \in \TSS^{d \times n}$ is the open tropical cone $\sep(A_{-i} \odot S^{(i)} \odot T^{(i)})$.
\end{theorem}
\begin{proof}
  By using an appropriate scaling matrix $S^{(i)}$, we can assume by Lemma~\ref{lem:transformation+normalized+row} that the absolute value of each entry in the $i$th row of $A$ is either $0$ or $\Zero$ and that this does not affect $\rho_i(\sep(A))$. 

  To simplify notation, we additionally assume that $i = d$. 
  Using Lemma~\ref{lem:resolve+balanced} and ordering the inequalities according to the partition from~\eqref{eq:index+sets+signs}, we get the system
  \begin{equation} \label{eq:ordered+resolved+system}
    \begin{aligned}
      y_d \oplus (y_1,\dots,y_{d-1}) \odot \trans{(a_{1,j},\dots,a_{d-1,j})} > \Zero & \text{ for } j \in J^+ \cup J^{\bullet} \\
   \ominus y_d \oplus (y_1,\dots,y_{d-1}) \odot \trans{(a_{1,j},\dots,a_{d-1,j})} > \Zero & \text{ for } j \in J^- \cup J^{\bullet}
    \end{aligned}
  \end{equation}
  Proposition~\ref{prop:elimination+system} implies that~\eqref{eq:ordered+resolved+system} has a solution $(y_1,\dots,y_{d-1},y_d) \in \TTpm^d$ if and only if
  \begin{equation} \label{eq:summed+inequalities}
    \begin{aligned}
    (y_1,\dots,y_{d-1}) \odot A_{-d} \odot T > \Zero 
    \end{aligned}
  \end{equation}
  has a solution $(y_1,\dots,y_{d-1}) \in \TTpm^{d-1}$. 
\end{proof}

  We can resolve the balanced entries which can occur in the product with $T^{(i)}$ by means of Theorem~\ref{thm:replace+balanced}.

\begin{corollary}[Fourier--Motzkin for strict inequalities with signed numbers] \label{cor:signed+fourier+motzkin}
  The $i$th coordinate projection of the open tropical cone $\rho_i(\sep(A))$ is the open tropical cone $\sep\left(\xi(A_{-i} \odot S^{(i)} \odot T^{(i)})\right)$.
\end{corollary}

\begin{example}
  We use the matrix
  \[
  A =
  \begin{pmatrix}
    3 & \ominus 1 & \ominus 4 \\
    3 & \ominus 0 & \ominus 2
  \end{pmatrix} \enspace ,
  \]
  from Figure~\ref{fig:first+hull}.

  Eliminating the first row yields
  \[
  \begin{pmatrix}
    3 & \ominus 0 & \ominus 2
  \end{pmatrix}
  \odot
  \begin{pmatrix}
    -3 & \Zero & \Zero \\
    \Zero & -1 & \Zero \\
    \Zero & \Zero & -4
  \end{pmatrix}
  \odot
  \begin{pmatrix}
    0 & 0 \\
    0 & \Zero \\
    \Zero & 0 
  \end{pmatrix}
  =
  \begin{pmatrix}
    0 & 0
  \end{pmatrix}
  \]
  This shows that $\sep^+(A)$ is not empty (as all entries of the remaining matrix are positive signed numbers), which can also be seen in Figure~\ref{fig:first+hull} using the duality in Theorem~\ref{thm:farkas+lemma}. 
  \end{example}

\begin{remark}
  The crucial difference to classical Fourier--Motzkin elimination happens in the treatment of balanced numbers occuring in the calculation.
  While classically in each step several variables could be eliminated at the same time we have to deal with their balanced
  left-overs. 
  For strict inequalities, Lemma~\ref{lem:resolve+balanced} provides a tool to resolve them by introducing two inequalities instead of one.
  We will see how to resolve them for non-strict inequalities in Proposition~\ref{prop:resolving+balanced+coefficiens+teq}.
\end{remark}

\begin{remark} \label{rem:connection+nonnegative+FM}
  The classical technique of Fourier-Motzkin for polytopes, see e.g.~\cite{ConfortiCornuejolsZambelli:2014} has already been successfully adapted to tropical linear inequality systems in~\cite{AllamigeonFahrenbergGaubertKatz:2014}.
  In the latter work, an algorithmic scheme to determine a projection of a tropical inequality system is described.
  In our Theorem~\ref{thm:balanced+fourier+motzkin}, we do not have the non-negatively constrained variables but allow arbitrary elements of $\TTpm$. 
  Classically, one can represent an unconstrained variable as the difference of a pair of non-negative variables. 
  Applying this technique to a system of the form $\trans{y} \odot A > \Zero$ with unconstrained $y \in \TTpm^n$ for a matrix $A \in \TTpm^{d \times n}$ yields the system
  \begin{equation}
    (\trans{u},\trans{v}) \odot
    \begin{pmatrix}
      A \\ \ominus A
    \end{pmatrix}\ >\ \Zero 
    \text{ with } u,v \in \Tgeq^d \enspace .
  \end{equation}
  Reordering terms with coefficients in $\Tlt$ to the other side of the inequality yields a system which allows to apply~\cite[Theorem 11]{AllamigeonFahrenbergGaubertKatz:2014}.
  However, the differences of non-negative variables are harder to resolve as there is no cancellation but it results in balanced entries.
  This makes our direct approach more tractable for unconstrained variables.
  Furthermore, we are also interested in the structure of the resulting inequalities in Section~\ref{sec:halfspaces}, whence our approach is more suitable for this setting. 
\end{remark}
%% $(\Tgeq^d)^2 \to \TTpm^d, (u,v) \mapsto u \cancsum (\ominus v)$ or $\Uncomp(u \ominus v)$.
%% $\TTpm^d \to (\Tgeq^d)^2, w \mapsto ((\max(w_i,\Zero),\ominus \min(w_i,\Zero))_i)$

The elimination procedure derived in the last section allows to prove the desired separation in $\TTpm^d$. 

\begin{proof}[Proof of Theorem~\ref{thm:farkas+lemma}]
  At first, we show the claim for $d = 1$. The set $\sep(A)$ is non-empty if and only if either all entries are positive or all entries are negative. Otherwise, we can select a balanced entry or a pair of entries with opposite sign by multiplication from the right.
  
  If $\sep(A)$ is not empty, then Theorem~\ref{thm:weak+duality} tells us that $\nnker(A)$ is empty. So, we assume that $\sep(A)$ is empty.
  As the scaling of the columns of $A$ does not change the sets $\sep(A)$ or $\nnker(A)$, we can assume that the absolute value of the entries in the last row of $A$ is $0$ or $\Zero$. 
  Let $T$ be the matrix from~\eqref{eq:incidence+matrix+elimination}.
  Then Theorem~\ref{thm:balanced+fourier+motzkin} shows that $\sep(A_{-d} \odot T)$ is empty.
  By induction, there is an element $z$ in $\nnker(A_{-d} \odot T)$.
  We show that $T \odot z \in \nnker(A)$.
  By definition of $T$, the elements in the $d$th row of $A \odot T$ are all in $\Tzero$.
  This implies that the $d$th row of $A \odot T \odot z$ is in $\Tzero$.
  Additionally, the choice of $z$ yields $A_{-d} \odot T \odot z \in \Tzero^{d-1}$.
  This finishes the proof. 
\end{proof}

\subsubsection{Non-strict inequalities}

While we mainly considered the strict inequalities as separators from the origin, in light of Theorem~\ref{thm:farkas+lemma}, we develop the elimination theory for non-strict inequalities with a view towards the representation of convex sets as intersection of closed halfspaces in Theorem~\ref{thm:signed+Minkowski+Weyl}.
Therefore, we add another coordinate to treat affine halfspaces, which only occured in Example~\ref{ex:open+halfspace+convex} so far.
Next, we derive the analogous statement to Theorem~\ref{thm:balanced+fourier+motzkin} for the relation `$\teq$' instead of `$>$'. 

Note that this is substantially more subtle than the case of strict inequalities.
We require the coefficient matrix of the inequalities to consist of signed numbers only.
However, as each elimination step can result in balanced coefficients, we again need a method to resolve those.
While this was rather easy for strict inequalities as shown in Theorem~\ref{thm:replace+balanced}, we need to develop a more technical and less explicit machinery in Proposition~\ref{prop:resolving+balanced+coefficiens+teq}. 

\todo[inline]{LV: Please fix. The matrix $T$ in the statement is only
  defined in the proof. The matrix $S$ is missing. For consistency,
  shall we make the claim for index $i$ instead of $d$?
  GL: I made several changes around this, please check. 
}

We recall the matrices $S^{(i)}$ from Lemma~\ref{lem:transformation+normalized+row} and $T^{(i)}$ from~\eqref{eq:incidence+matrix+elimination} with $d$ replaced by $d+1$ and $J^{\bullet} = \emptyset$.

\begin{theorem}[Fourier--Motzkin for non-strict inequalities with signed numbers] \label{thm:non+strict+signed+fourier+motzkin}
  The $i$th coordinate projection of the set
  \begin{equation}
  \SetOf{y \in \TTpm^d}{(0,\trans{y}) \odot A \teq \Zero}
  \end{equation}
  for $A \in \TTpm^{(d+1) \times n}$ is the set
  \begin{equation} \label{eq:non+strict+system+projection}
    \SetOf{z \in \TTpm^{d-1}}{(0,\trans{z}) \odot A_{-i}\odot S^{(i)} \odot T^{(i)} \teq \Zero} \enspace .
  \end{equation}
\end{theorem}
%\end{proposition}
\begin{proof}
  Using Lemma~\ref{lem:transformation+normalized+row}, we can assume that the absolute value of each entry in the $i$th row of $A$ is either $0$ or $\Zero$ by using an appropriate scaling matrix $S^{(i)}$.
  %% This can be achieved by multiplying each column of $A$ indexed by $J^+ \cup J^-$ with the inverse of the entry in its $d$th row.
  %% The inequality $\trans{y} \odot A \teq \Zero$ remains valid in this transformation.

  To simplify notation, we additionally assume that $i = d$. 
Ordering the inequalities according to the distinction from~\eqref{eq:index+sets+signs}, we get the system
  \begin{equation} \label{eq:ordered+teq+system}
    \begin{aligned}
      y_d \oplus (0,y_1,\dots,y_{d-1}) \odot \trans{(a_{0,j},a_{1,j},\dots,a_{d-1,j})} \teq \Zero & \text{ for } j \in J^+ \\
   \ominus y_d \oplus (0,y_1,\dots,y_{d-1}) \odot \trans{(a_{0,j},a_{1,j},\dots,a_{d-1,j})} \teq \Zero & \text{ for } j \in J^- 
    \end{aligned}
  \end{equation}
  Proposition~\ref{prop:elimination+system+non+strict} implies that~\eqref{eq:ordered+teq+system} has a solution $(0,y_1,\dots,y_{d-1},y_d) \in \TTpm^{d+1}$ if and only if
  \begin{equation} \label{eq:summed+teq+inequalities}
    \begin{aligned}
    (0,y_1,\dots,y_{d-1}) \odot A_{-d} \odot S^{(d)}\odot T^{(d)} \teq \Zero 
    \end{aligned}
  \end{equation}
  has a solution $(0,y_1,\dots,y_{d-1}) \in \TTpm^{d}$. 
\end{proof}

\begin{example}
  We will see how to obtain an exterior description by closed halfspaces in Theorem~\ref{thm:signed+Minkowski+Weyl}.
  To determine the exterior description of the one dimensional line segment from $\ominus 0$ to $1$ in $\TTpm$, one can eliminate $x_1$ and $x_2$ from the system
  \begin{subequations}
    \begin{align}
      \ominus 0 \odot x_1 \oplus 1 \odot x_2 \ominus z \teq \Zero \label{eq:sub+1} \\
      0 \odot x_1 \ominus 1 \odot x_2 \oplus z \teq \Zero \label{eq:sub+2} \\
      0 \odot x_1 \oplus 0 \odot x_2 \ominus 0 \teq \Zero \label{eq:sub+3} \\
      \ominus 0 \odot x_1 \ominus 0 \odot x_2 \oplus 0 \teq \Zero \label{eq:sub+4} \\
      0\odot x_1 \teq \Zero \label{eq:sub+5} \\
      0 \odot x_2 \teq \Zero \label{eq:sub+6}
    \end{align}
  \end{subequations}
  Eliminating $x_1$ yields
  \begin{subequations} \label{eq:x1+eliminated}
    \begin{align}
       \bullet 1 \odot x_2 \bullet z &\teq \Zero \qquad \text{ from } \eqref{eq:sub+1} \& \eqref{eq:sub+2} \\
       1 \odot x_2 \ominus z \ominus 0 &\teq \Zero \qquad \text{ from } \eqref{eq:sub+1} \& \eqref{eq:sub+3} \label{eq:sub+a}\\
       1 \odot x_2 \ominus z &\teq \Zero \qquad \text{ from } \eqref{eq:sub+1} \& \eqref{eq:sub+5} \\
       \ominus 1 \odot x_2 \oplus z \oplus 0 &\teq \Zero \qquad \text{ from } \eqref{eq:sub+4} \& \eqref{eq:sub+2} \label{eq:sub+b}\\
       \bullet 0 \odot x_2 \bullet 0 &\teq \Zero \qquad \text{ from } \eqref{eq:sub+4} \& \eqref{eq:sub+3}\\
       \ominus 0 \odot x_2 \oplus 0 &\teq \Zero \qquad \text{ from } \eqref{eq:sub+4} \& \eqref{eq:sub+5} \label{eq:sub+c}\\
       0 \odot x_2 &\teq \Zero \qquad \text{ from } \eqref{eq:sub+6} \label{eq:sub+d}
    \end{align}
  \end{subequations}
  From further elimination of $x_2$ we get by ignoring redundant inequalities of~\eqref{eq:x1+eliminated}
  \begin{subequations}
    \begin{align}
      \bullet z \bullet 0 &\teq \Zero \qquad \text{ from } \eqref{eq:sub+a} \& \eqref{eq:sub+b} \\
      \ominus (-1) \odot z \oplus 0 &\teq \Zero \qquad \text{ from } \eqref{eq:sub+a} \& \eqref{eq:sub+c}  \\
      (-1) \odot z \oplus (-1) &\teq \Zero \qquad \text{ from } \eqref{eq:sub+d} \& \eqref{eq:sub+b}  \\
      0 &\teq \Zero \qquad \text{ from } \eqref{eq:sub+d} \& \eqref{eq:sub+c} 
    \end{align}
  \end{subequations}
  This yields the exterior description $z \teq \ominus 0$ and $z \teq \ominus 1$. 
\end{example}

\subsection{On the different forms of tropical LP}
For (regular) linear programs, solving all standard forms such as $Ax=b$,
$x\ge 0$ or $Ax\le b$ are polynomial-time equivalent. For example, consider
a system in the second form with $A\in \mathbb{R}^{m\times n}, b \in \mathbb{R}^{m}$. By using
variables  $x',x''\in \mathbb{R}^n$ and $s\in \mathbb{R}^m$, we can write a system in
the first form as $(A,-A,I_m)(x',x'',s)=b$, $(x',x'',s)\ge 0$. From a
feasible solution $(x',x'',s)$, we can recover a feasible solution
$x=x'-x''$ to the original system.

The situation appears to be different for tropical LP. Consider
systems of the form $A\odot x\beq b$, $x\geq \Zero$, $x\in \TTpm$ and $A\odot
x\seq b$, $x\in \TTpm$; in both cases, we need to find a nonzero solution.
The first system asks for
$\nnker \begin{pmatrix}
    A & \ominus b \\
    0 & \ominus 0
  \end{pmatrix}
\neq \emptyset$, and is
essentially equivalent to solving mean payoff games
\eqref{eq:tropical+linear+inequality+system}. Deciding whether this
system is feasible is in the complexity class NP$\cap$coNP; this can
be seen from Farkas' lemma (Theorem~\ref{thm:farkas+lemma}), and an
additional argument showing that if one of the respective systems is
feasible, there is a solution of bit complexity bounded in the bit
complexity of the input.

In contrast, for finding a nonzero solution to the system $A\odot
x\seq b$,  $x\in \TTpm$, showing membership in NP is easy, but membership in coNP is unclear. We do not obtain a
variant of Farkas' lemma for this system via the above Fourier-Motzkin
elimination. The construction analogous to the classical case would
transform it to $(A,\ominus A,I_m)\odot (x',x'',s)\beq b$,
$(x',x'',s)\geq \Zero$.  However, the solution $x=x'\ominus x''$ may
contain balanced entries, and it may not be possible to resolve 
these entries to get a feasible solution.

\begin{example}
  We consider the system
  \begin{equation*}
    \begin{pmatrix}
      0 & 0 \\
      \ominus 0 & 0 \\
      0 & \ominus 0 \\
      \ominus 0 & \ominus 0
    \end{pmatrix}
    \odot
    \begin{pmatrix} x_1 \\ x_2 \end{pmatrix} \seq
    \begin{pmatrix} \ominus 0 \\ \ominus 0 \\ \ominus 0 \\ \ominus 0 \end{pmatrix} \enspace .
  \end{equation*}
  It does not have a signed solution.
  However the extended system
  \begin{equation*}
    \begin{pmatrix}
      0 & 0 & \ominus 0 & \ominus 0 & 0 & \Zero & \Zero & \Zero \\
      \ominus 0 & 0 & 0 & \ominus 0 & \Zero & 0 & \Zero & \Zero \\
      0 & \ominus 0 & \ominus 0 & 0 & \Zero & \Zero & 0 & \Zero \\
      \ominus 0 & \ominus 0 & 0 & 0 & \Zero & \Zero & \Zero & 0
    \end{pmatrix}
    \odot
    \begin{pmatrix} x'_1 \\ x'_2 \\ x''_1 \\ x''_2 \\ s_1 \\ s_2 \\ s_3 \\ s_4 \end{pmatrix} \beq \begin{pmatrix} \ominus 0 \\ \ominus 0 \\ \ominus 0 \\ \ominus 0 \end{pmatrix}\, ,\qquad x',x'',s \geq \Zero
  \end{equation*}
  has the solution $(0,0,0,0,0,0,0,0)$. 
\end{example}

It is unclear whether the system  $A\odot
x\seq b$,  $x\in \TTpm$ can be reduced to tropical LPs with
nonnegative variables, or if it constitutes a more general problem.
\section{Description by halfspaces} \label{sec:halfspaces}

%We then use the separation results and Fourier-Motzkin elimination to describe signed tropically convex sets as intersections of tropical halfspaces. 

An important property of classical polytopes is the duality between the representation as convex hull and as intersection of finitely many halfspaces.
This is more subtle for tropical polytopes over $\TTpm$.
While we establish a description as intersection of open tropical
halfspaces (Theorem~\ref{thm:outer+hull+open+halfspaces}) containing a
set of points, we need additional properties to formulate a
Minkowski-Weyl theorem (Theorem~\ref{thm:signed+Minkowski+Weyl}). All
proofs are deferred to Section~\ref{sec:halfspace+proofs}.

\smallskip

We saw already in Examples~\ref{ex:open+halfspace+convex} that open halfspaces are convex. 
 The \emph{outer hull} of a set $M \subseteq \TTpm^d$ is the intersection of its containing open halfspaces 
\begin{equation} \label{eq:intersection+open+halfspaces}
  \bigcap_{M \subseteq \HC^+(a)} \HC^+(a) \enspace .
\end{equation}

This construction yields the same as our key Definition~\ref{def:inner+hull}.

\begin{theorem} \label{thm:outer+hull+open+halfspaces}
  The outer and the inner hull of a matrix $A \in \TTpm^{d \times n}$ coincide, i. e., $\tconv(A) = \bigcap_{A \subseteq \HC^+(v)} \HC^+(v)$, where we identify $A$ with the set of its columns.
\end{theorem}

\begin{figure}[ht]
  \begin{tikzpicture}

  \Koordinatenkreuz{-2.1}{2.1}{-1.2}{2.3}{$x_1$}{$x_2$}

  \coordinate (A) at (1.4,1.4){};
  \coordinate (B) at (-1.7,-0.9){};
  \coordinate (C) at (-1.0,-0.5){};
  \coordinate (AB) at (-1.7,1.4){};
  \coordinate (AC1) at (-1.0,1.0){};
  \coordinate (AC2) at (1.0,1.0){};
  \coordinate (BC) at (-1.3,-0.5){};
  \coordinate (BCneg) at ($-1*(BC)$){};
  \coordinate (ABneg) at ($-1*(AB)$){};

  \draw[fill,Linesegment,opacity=0.2] (A) -- (AB) -- (B) -- (BC) -- (C) -- (AC1) -- (AC2) -- cycle;

  \node[Endpoint] at (A){};
  \node[Endpoint] at (B){};
  \node[Endpoint] at (C){};

  \newcommand\meps{0.1}
  \coordinate (apex11) at ($(AC1) + (\meps,-2*\meps)$);
  \coordinate (apex12) at ($(AC2) + (-\meps,-2*\meps)$);
  \coordinate (apex13) at ($-1*(AC2) + (\meps,2*\meps)$);
  \coordinate (apex14) at ($(apex12)+(1.5,1.5)$);
  \coordinate (apex15) at ($(apex13)+(-1.3,-1.3)$);
  
  \draw[] (apex15) -- (apex13) -- (apex11) -- (apex12) -- (apex14);

  \coordinate (apex21) at ($(BC) + (0,-\meps)$);
  \coordinate (apex22) at ($(BC -| BCneg) + (0,-\meps)$);
  \coordinate (apex23) at ($(BCneg) + (0,+\meps)$);
  \coordinate (apex24) at ($(apex21)+(-1.1,-1.1)$);
  \coordinate (apex25) at ($(apex23)+(1.1,1.1)$);

  \draw[] (apex24) -- (apex21) -- (apex22) -- (apex23) -- (apex25);

  \coordinate (apex31) at ($(AB) + (-\meps,\meps)$);
  \coordinate (apex32) at ($(AB -| ABneg) + (\meps,\meps)$);
  \coordinate (apex33) at ($(AB |- ABneg) + (-\meps,-\meps)$);
  \coordinate (apex34) at ($(apex32)+(0.6,0.6)$);
  \coordinate (apex35) at ($(apex33)+(-0.6,-0.6)$);

 \draw[] (apex35) -- (apex33) -- (apex31) -- (apex32) -- (apex34);
  
\end{tikzpicture}
  \caption{Approximation of a triangle by open halfspaces}
\end{figure}
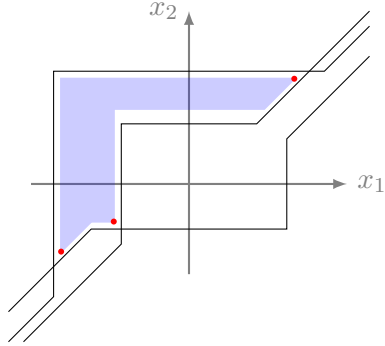

As the latter does not give a finite description, one is usually more interested in closed halfspaces. For a vector $(a_0,a_1,\dots,a_d) \in \TTpm^{d+1}$, we define a \emph{closed (signed) tropical halfspace} by
\begin{equation}
  \Hclosed^+(a) = \SetOf{x \in \TTpm^d}{a \odot \begin{pmatrix} 0 \\ x \end{pmatrix} \in \Tgeq \cup \Tzero} \enspace .
\end{equation}

\begin{lemma} \label{lem:relation+open+closed+halfspaces}
  The closed signed tropical halfspace $\Hclosed^+(a)$ is the topological closure of the open signed halfspace $\HC^+(a)$.
\end{lemma}

Observe that the former statement is wrong for inequalities with balanced numbers as coefficients. 

\begin{remark} \label{rem:closed+tropical+halfspaces+hull}
  Closed tropical halfspaces are not as suitable for the hull construction in Equation~\eqref{eq:intersection+open+halfspaces} as open tropical halfspaces. 
  The inner hull of $M = \{(\ominus 1, 1), (1, \ominus 1)\}$ should contain the origin $\Zero$.
  However, the closed halfspace $x_1 \oplus x_2 \geq \ominus 0$ contains those two points but not the origin.
  Taking the analogous intersection as in~\eqref{eq:intersection+open+halfspaces} with closed tropical halfspaces for $M$ yields again $M$.
  Note that this is the same as the intersection of the balanced image $M \odot \Delta_2$ with $\TTpm^2$.
\end{remark}

Example~\ref{ex:coordinate+hyperplane} shows that such a closed signed tropical halfspace is in general not tropically convex.
However, a finite intersection of such halfspaces can be tropically convex, as Figure~\ref{fig:closed+halfspaces+line+segment} shows. 

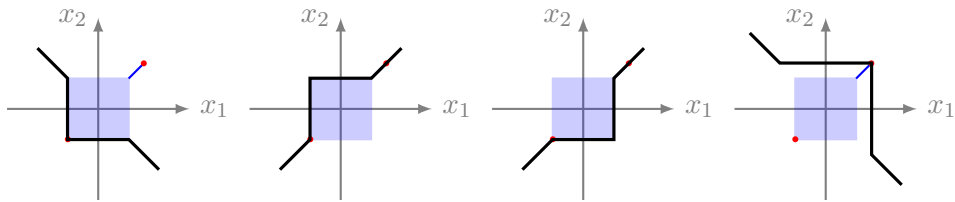
\begin{figure} 
  \newcommand\basicsetforext{
  \Koordinatenkreuz{-3}{3}{-3}{3}{$x_1$}{$x_2$}

  \coordinate (center) at (1.5,1.5){};
  \coordinate (centerneg) at ($-1*(center)$){};
  \coordinate (pp) at (1,1){};
  \coordinate (mp) at (-1,1){};
  \coordinate (pm) at (1,-1){};
  \coordinate (mm) at (-1,-1){};
  
  \draw[Linesegment] (center) -- (pp);
  \draw[fill,Linesegment,opacity=0.2] (pp) -- (mp) -- (mm) -- (pm) -- cycle;

  \node[Endpoint] at (center){};
  \node[Endpoint] at (mm){};
}
\newcommand\neps{0.01}
\newcommand\raylen{1}
\newcommand\scalefactor{0.404}

  \begin{tikzpicture}[scale = \scalefactor]
    \basicsetforext
  \draw[ClosedHalfspace] ($(mp) + (-\raylen,\raylen)$) -- ($(mp) + (-\neps,\neps)$) -- ($(mm) + (-\neps,-\neps)$) -- ($(pm) + (\neps,-\neps)$) -- ($(pm) + (\raylen,-\raylen)$);
\end{tikzpicture}
  \begin{tikzpicture}[scale = \scalefactor]
    \basicsetforext
  \draw[ClosedHalfspace] ($(pp) + (\raylen,\raylen)$) -- ($(pp) + (\neps,\neps)$) -- ($(mp) + (-\neps,\neps)$) -- ($(mm) + (-\neps,-\neps)$) -- ($(mm) + (-\raylen,-\raylen)$);
\end{tikzpicture}
  \begin{tikzpicture}[scale = \scalefactor]
    \basicsetforext
  \draw[ClosedHalfspace] ($(pp) + (\raylen,\raylen)$) -- ($(pp) + (\neps,\neps)$) -- ($(pm) + (\neps,-\neps)$) -- ($(mm) + (-\neps,-\neps)$) -- ($(mm) + (-\raylen,-\raylen)$);
\end{tikzpicture}
  \begin{tikzpicture}[scale = \scalefactor]
    \basicsetforext
  \draw[ClosedHalfspace] ($(center -| centerneg) + (-\raylen,\raylen)$) -- ($(center -| centerneg) + (-\neps,\neps)$) -- ($(center) + (\neps,\neps)$) -- ($(center |- centerneg) + (\neps,-\neps)$) -- ($(center |- centerneg) + (\raylen,-\raylen)$);
\end{tikzpicture}
  \caption{Exterior description of $\tconv((\ominus 1, \ominus 1),(2,2))$ by closed halfspaces. }
  \label{fig:closed+halfspaces+line+segment}
\end{figure}

Hence, to arrive at a Minkowski-Weyl theorem for tropical polytopes over $\TTpm$, one has to adapt the condition in the characterization of finite intersections of closed tropical halfspaces. 

\begin{theorem} \label{thm:signed+Minkowski+Weyl}
  For each finite set $V \subset \TTpm^d$, there are finitely many closed tropical halfspaces $H$ such that $\tconv(V)$ is the intersection of the halfspaces.

  \smallskip

  For each finite set $H$ of closed halfspaces, whose intersection $M$ is tropically convex, there is a finite set of points $V \in \TTpm^d$ such that $M = \tconv(V)$. 
\end{theorem}

\begin{corollary}
  A tropically convex set is the intersection of its containing closed halfspaces. 
\end{corollary}

\begin{remark}
  The crucial difference with Theorem~\ref{thm:outer+hull+open+halfspaces} is that for open tropical halfspaces the generators are enough, while for closed tropical halfspaces, we have to take the whole set into account. 
\end{remark}

\subsection{Proofs}\label{sec:halfspace+proofs}

To prove that the inner hull and the outer hull coincide, we mainly have to use the separation results from Section~\ref{sec:Farkas+FM}. 

\begin{proof}[Proof of Theorem~\ref{thm:outer+hull+open+halfspaces}]
  %% Let $I = \{ i \in [d] \mid \bar{p}_i = \ominus \bar{q}_i \}$.
  %% By symmetry in $p$ and $q$ (?!) we can restrict to one point in $\zeta_d(p,q)$. We get $v_0 \oplus \bigoplus_{i\in [d] \setminus I} (p_i \oplus q_i) \odot v_i \oplus \bigoplus_{I \setminus \{d\}} p_i \odot v_i$
  The inclusion $\tconv(A) \subseteq \bigcap_{A \subseteq \HC^+(v)} \HC^+(v)$ follows by combining Proposition~\ref{prop:immediate+properties+convex+sets}.\ref{item:intersection+convex+sets} and Example~\ref{ex:open+halfspace+convex}. 

  \smallskip
  
  For the other inclusion, assume that there is a point
  \[
  z \in \bigcap_{A \subseteq \HC^+(v)} \HC^+(v) \setminus \tconv(A) \enspace .
  \]
  By Proposition~\ref{prop:equivalence+hull+solution},
  we get that $\nnker(B) = \emptyset$, where
  \[
  B =
  \begin{pmatrix}
    0 & \ominus 0 \\
    A & \ominus z
  \end{pmatrix} \enspace .
  \]
  Theorem~\ref{thm:farkas+lemma} implies that $\sep(B) \neq \emptyset$.
  Hence, there is a separator $(u_0,\bar{u}) = (u_0,u_1,\ldots,u_d) \in \TTpm^{d+1}$ with $u_0 \oplus \trans{\bar{u}} \odot a^{(j)} > \Zero$ for all columns $a^{(j)}$ of $A$, and $\ominus u_0 \oplus \ominus \trans{\bar{u}} \odot z > \Zero \Leftrightarrow u_0 \oplus \trans{\bar{u}} \odot z < \Zero$. This means that the columns of $A$ lie in the halfspace $\HC^+(u)$ but $z$ does not. This contradicts the choice of $z$ in $\bigcap_{A \subseteq \HC^+(v)}\HC^+(v)$.
\end{proof}

Showing that a closed tropical halfspace is obtained as the closure of an open tropical halfspace follows from a simple limit argument.
We use the notation $[d]_0 = \{0,1,2,\dots,d\}$.

\begin{proof}[Proof of Lemma~\ref{lem:relation+open+closed+halfspaces}]
  Let $z \in \Tgeq^d$ be an element of $\Hclosed^+(a) \setminus \HC^+(a)$.
  Set
  \[
  J = \argmax\SetOf{|c_j \odot z_j|}{j \in [d]_0} \enspace ,
\]
  where $z_0 = 0$, and let $\ell \in J$ with $c_{\ell} \odot z_{\ell} > \Zero$.
  Denoting the $k$th unit vector $(0,\dots,0,1,0,\dots,0) \in \RR^d$ by $e_k$ for $k \in [d]$ and $e_0 = (-1, \dots, -1)$, we define a sequence
  \[
  y^{(m)} = z + \frac{1}{m}\cdot e_{\ell} \enspace .
  \]
  The sequence converges to $z$ but each element of the sequence is an element of $\HC^+(a)$. 
\end{proof}

We will use the Fourier-Motzkin version for the relation `$\teq$' (Theorem~\ref{thm:non+strict+signed+fourier+motzkin}) to deduce an exterior description of a tropical polytope. 
However, the system describing the projection in~\eqref{eq:non+strict+system+projection} may contain balanced coefficients.
We address this issue in the next statement. 
It is an existence argument statement that a balanced coefficient can be replaced by a signed coefficient, see also Figure~\ref{fig:resolving+balanced+coeff}.

\begin{proposition}[Resolving balanced
  coefficients] \label{prop:resolving+balanced+coefficiens+teq}
  Let $M$ be a tropically convex set and $c = (c_0,c_1,\dots,c_d) \in \TSS^d$ with $c_i \in \Tzero$ for some $i \in [d]_0$ such that
  \[
  M \subseteq H(c) = \SetOf{x \in \TTpm^d}{c \odot \begin{pmatrix} 0 \\ x \end{pmatrix} \teq \Zero} \enspace .
  \]
  Then there is an element $b \in \Uncomp(c_i)$ such that $M$ is contained in 
  \[
 \SetOf{x \in \TTpm^d}{(c_0,\dots,c_{i-1},b,c_{i+1},\dots,c_d) \odot \begin{pmatrix} 0 \\ x \end{pmatrix} \teq \Zero} \enspace .
  \]  
\end{proposition}
\begin{proof}
  If $i = 0$, then we can set $b = |c_i|$.

  \smallskip

  Assuming without loss of generality that $i = d$, we set $c_{-d} = (c_0,c_1,\dots,c_{d-1})$.
  Fix $u \in \TTpm^{d-1}, v \in \TTpm$ such that $(u,v) \in M$.
  We define $w = w(u) = c_{-d} \odot \begin{pmatrix} 0 \\ u \end{pmatrix}$.

  %% Let
  %% \[
  %% \underline{\lambda} = \min_{(u,v) \in M, v < \Zero} \argmax_{w \oplus \lambda \odot v \teq \Zero }
  %% \]
  
  If $v \in \Tlt$, then 
  \begin{equation} \label{eq:argmax+lambda}
  \lambda(u,v) = \argmax\SetOf{\lambda \in \Uncomp(c_d)}{w \oplus \lambda \odot v \teq \Zero } \enspace .
  \end{equation}
  Let
  \begin{equation} \label{eq:min+max+lambda}
  \underline{\lambda} = \min\SetOf{\lambda(u,v)}{(u,v) \in M, v < \Zero} \enspace .
  \end{equation}
  If $v \in \Tgt$, then 
  \begin{equation} \label{eq:argmin+lambda}
  \lambda(u,v) = \argmin\SetOf{\lambda \in \Uncomp(c_d)}{w \oplus \lambda \odot v \teq \Zero } \enspace .
  \end{equation}
  Let
  \begin{equation} \label{eq:max+min+lambda}
  \overline{\lambda} = \max\SetOf{\lambda(u,v)}{(u,v) \in M, v > \Zero} \enspace .
  \end{equation}
  
  We derive a contradiction to the convexity of $M$, if $\overline{\lambda} > \underline{\lambda}$.
  
  Let $(p,q)$ and $(r,s)$ attain $\underline{\lambda}$ and $\overline{\lambda}$, respectively.
  In particular, we have $q < \Zero$ and $s > \Zero$. 
  The inequality $\overline{\lambda} > \underline{\lambda}$ implies that $\overline{\lambda} > \ominus |c_d|$ and that $\underline{\lambda} < |c_d|$.
  We can assume that $w(p)$ and $w(r)$ are signed numbers, as we otherwise can replace them by their absolute value without changing the admissible values of $\lambda$ in~\eqref{eq:argmax+lambda} and~\eqref{eq:argmin+lambda}. 
  
  Now, the construction of $\underline{\lambda}$ in~\eqref{eq:argmax+lambda} implies that $w(p) = \ominus \underline{\lambda} \odot q$ and~\eqref{eq:argmin+lambda} yields $w(r) = \ominus \overline{\lambda} \odot s$.
  This implies
  \begin{equation} \label{eq:comparison+ratio+lambda}
  w(r) \odot s^{\odot-1} = \ominus \overline{\lambda} < \ominus \underline{\lambda} = w(p) \odot q^{\odot-1} \enspace .
  \end{equation}
  We consider the point in the convex combination of $(p,q)$ and $(r,s)$ given by
  \[
  z = (\ominus q^{\odot-1} \oplus s^{\odot-1})^{\odot-1} \odot (\ominus q^{\odot-1}\odot p \oplus s^{\odot-1}\odot r,\Zero) \enspace .
  \]
  Then
  \begin{equation*}
    \begin{aligned}
      c \odot \begin{pmatrix} 0 \\ z \end{pmatrix} &= c \odot \begin{pmatrix} 0 \\ (\ominus q^{\odot-1} \oplus s^{\odot-1})^{\odot-1} \odot (\ominus q^{\odot-1}\odot p \oplus s^{\odot-1}\odot r) \\ \Zero \end{pmatrix} = \\
      &= (\ominus q^{\odot-1} \oplus s^{\odot-1})^{\odot-1} \odot \begin{pmatrix} \ominus q^{\odot-1}\odot c_{-d} \odot \begin{pmatrix} 0 \\ p \end{pmatrix} \oplus s^{\odot-1}\odot c_{-d} \odot \begin{pmatrix} 0 \\ r \end{pmatrix} \\ \Zero \end{pmatrix} = \\
      &= (\ominus q^{\odot-1} \oplus s^{\odot-1})^{\odot-1} \odot \begin{pmatrix} \ominus q^{\odot-1}\odot w(p) \oplus s^{\odot-1}\odot w(r) \\ \Zero \end{pmatrix}
    \end{aligned}
  \end{equation*}
  By~\eqref{eq:comparison+ratio+lambda}, this implies $ c \odot \begin{pmatrix} 0 \\ z \end{pmatrix} < \Zero$.
  As $M$ is convex, this yields $z \in M \subseteq H(c)$, a contradiction. 

  \smallskip
  
  So, we can conclude that $\overline{\lambda} \leq \underline{\lambda}$. Fix any $b$ in the interval $[\overline{\lambda}, \underline{\lambda}] \neq \emptyset$. 
  Let $(u,v) \in M$ with $v < \Zero$. Then
  \[
  (c_{-d},b) \odot \begin{pmatrix} 0 \\ u \\ v \end{pmatrix} = w(u) \oplus b \odot v \teq w(u) \oplus \underline{\lambda} \odot v \teq w(u) \oplus \lambda(u,v) \odot v \teq \Zero \enspace ,
  \]
  where we use $b \leq \underline{\lambda}$,~\eqref{eq:min+max+lambda} and Lemma~\ref{lem:basic+properties+teq}.\ref{item:teq+affine+monotony}.
  The proof for $v > \Zero$ goes analogously. 
\end{proof}

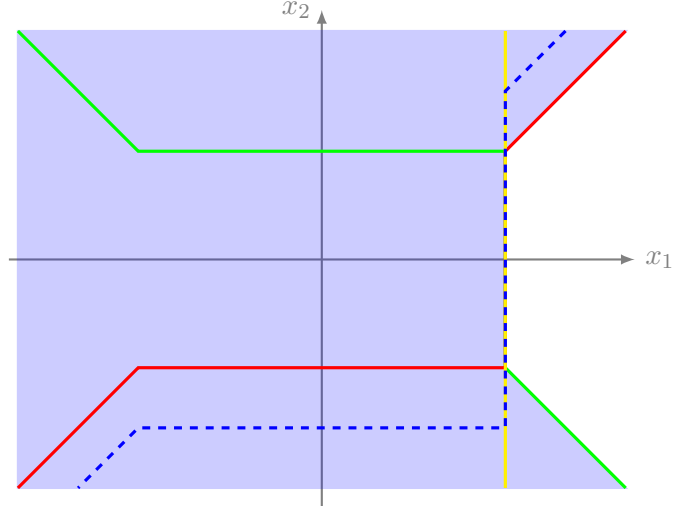
\begin{figure}[ht]
  \begin{tikzpicture}[scale = 0.4]

  \newcommand\addray{4}
  \newcommand\shiftbalanced{2}

  \Koordinatenkreuz{-10.4}{10.4}{-8.2}{8.3}{$x_1$}{$x_2$}

  \coordinate (apex1) at (6.1,3.6);
  \coordinate (apex1neg) at ($-1*(apex1)$);

  \coordinate (apex2) at (apex1 |- apex1neg);
  \coordinate (apex2neg) at ($-1*(apex2)$);

  \coordinate (apex3) at ($(apex2) + (0,-\shiftbalanced)$);
  \coordinate (apex3neg) at ($-1*(apex3)$);

  \draw[fill,Linesegment,opacity=0.2] ($(apex2 |- apex2neg) + (\addray,\addray)$) -- ($(apex1 -| apex1neg) + (-\addray,\addray)$) -- ($(apex2 -| apex2neg) + (-\addray,-\addray)$) -- ($(apex1 |- apex1neg) + (\addray,-\addray)$) -- (apex2) -- (apex1) -- cycle;
  
  \draw[ClosedHalfspace, green] ($(apex1 |- apex1neg) + (\addray,-\addray)$) -- (apex1 |- apex1neg) -- (apex1) -- (apex1 -| apex1neg) -- ($(apex1 -| apex1neg) + (-\addray,\addray)$);

  \draw[ClosedHalfspace, red] ($(apex2 |- apex2neg) + (\addray,\addray)$) -- (apex2 |- apex2neg) -- (apex2) -- (apex2 -| apex2neg) -- ($(apex2 -| apex2neg) + (-\addray,-\addray)$);

  \draw[ClosedHalfspace, yellow] ($(apex2 |- apex2neg) + (0,\addray)$) -- ($(apex1 |- apex1neg) + (0,-\addray)$);

  \draw[ClosedHalfspace, dashed, blue] ($(apex3 |- apex3neg) + (\addray - \shiftbalanced,\addray - \shiftbalanced)$) -- (apex3 |- apex3neg) -- (apex3) -- (apex3 -| apex3neg) -- ($(apex3 -| apex3neg) + (-\addray +\shiftbalanced,-\addray + \shiftbalanced)$);
  
\end{tikzpicture}
  
  \caption{Finding a containing halfspace without balanced coefficients}
  \label{fig:resolving+balanced+coeff}
\end{figure}

\begin{example}
  A pathological example for the last statement arises from resolving the tautological relation $\bullet (-1) \odot x \oplus \bullet (-1) \odot y \oplus \bullet 0 \teq \Zero$. One obtains the chain of relations
  \[
  \begin{aligned}
    \bullet (-1) \odot x \oplus \bullet (-1) \odot y \oplus \bullet 0 &\teq \Zero \Leftrightarrow \\
    \bullet (-1) \odot x \oplus \bullet (-1) \odot y \oplus 0 &\teq \Zero \Leftrightarrow \\
    \bullet (-1) \odot x \oplus \Zero \odot y \oplus 0 &\teq \Zero \Leftrightarrow \\
    \Zero \odot x \oplus \Zero \odot y \oplus 0 &\teq \Zero \enspace .
  \end{aligned}
  \]
\end{example}

\begin{remark}
  Indeed, any value of $b \in \Uncomp(c_d)$ in the former proof could occur as Figure~\ref{fig:resolving+balanced+coeff} shows. Any part of the dashed line without opposite (with respect to the origin) points is a tropical line segment. 
%  \todo[inline]{GL: Add example of line segment which is tight for some $b$. }
\end{remark}

\begin{example}
  The shaded area in Figure~\ref{fig:resolving+balanced+coeff} shows the feasible region
  \[
  \SetOf{(x,y) \in \TTpm^2}{\ominus (-1) \odot x_1 \oplus (\bullet 2) \odot x_2 \oplus 0 \teq \Zero} \enspace .
  \]
  The red line marks the inequality $\ominus(-1)\odot x_1 \oplus 2 \odot x_2 \oplus 0 \teq \Zero$, while the green line marks the inequality $\ominus(-1)\odot x_1 \oplus (\ominus 2) \odot x_2 \oplus 0 \teq \Zero$.
  The yellow line corresponds to the inequality $\ominus (-1) \odot x_1 \oplus 0 \teq \Zero$.
  The blue dashed line interpolates between these three possible extreme closed halfspaces contained in the feasible region. 
\end{example}

Our proof of Theorem~\ref{thm:signed+Minkowski+Weyl} is based on eliminating variables from the canonical exterior description~\eqref{eq:hull+union+intervals}.
For using those halfspaces, we need to show the additional requirement of tropical convexity for their intersection.

\begin{lemma} \label{lem:convexity+cone+preimage}
  The set
  %%   \[
  %% \SetOf{(x,z) \in \TTpm^{n+d}}{A \odot x \beq z,\; \bigoplus_{j \in [n]}x = 0,\; x \geq \Zero} 
  %% \]
  \begin{equation} \label{eq:convexity+cone}
  \SetOf{(x,z) \in \TTpm^{n+d}}{A \odot x \beq z,\; x \geq \Zero} 
  \end{equation}
  is tropically convex. 
%  In particular, the tropical convex hull $\tconv(A)$ is tropically convex and $\tconv(\tconv(A)) = \tconv(A)$. 
\end{lemma}
\begin{proof}
  It is enough to show that, for fixed $a \in \TTpm^n$, the set
  \[
  H = \SetOf{(x,z) \in \TTpm^{n+1}}{a \odot x \beq z,\; x \geq \Zero}
  \]
  is tropically convex, then the claim follows from
  Proposition~\ref{prop:immediate+properties+convex+sets}\eqref{item:intersection+convex+sets}.
Let $(p,q), (r,s)\in H$,  and $\lambda\in \Tgeq$, $\lambda\le
0$. We need to show that $\Uncomp(p\oplus\lambda\odot r, q\oplus\lambda \odot
s)\subseteq H$.

Note that since $p,r \geq \Zero$, we have $p\oplus\lambda\odot r\in
\TTpm$ and therefore $\Uncomp(p\oplus\lambda\odot r)=\{p\oplus\lambda\odot r\}$.
By Lemma~\ref{lem:basic+properties+beq}\eqref{item:beq+U}, we have that $q\in \Uncomp(a\odot p)$ and $s\in \Uncomp(a\odot r)$. Using
Lemma~\ref{lem:beq+comb}\eqref{item:u+contain}, we see that 
\[
\Uncomp(q\oplus\lambda \odot s)\subseteq \Uncomp((a\odot
p)\oplus\lambda\odot(a\odot r))=\Uncomp(a\odot (p\oplus\lambda\odot
r)) \enspace,
\]
completing the proof.
\end{proof}

\begin{proof}[Proof of Theorem~\ref{thm:signed+Minkowski+Weyl}]
  Equation~\eqref{eq:hull+union+intervals} provides a description by halfspaces involving additional variables.
  The convex hull of $V$ is the set of those $z \in \TTpm^d$ for which there is an $x \in \TTpm^n$ with
  \begin{equation} \label{eq:equation+system+inner+hull}
    \begin{pmatrix}
      A & \ominus \Id_d & \Zero \\
      \ominus A & \Id_d & \Zero \\
      0 & \Zero_d & \ominus 0 \\
      \ominus 0 & \Zero_d & 0 \\
      \Id_n & \Zero_d & \Zero
    \end{pmatrix}
    \odot
    \begin{pmatrix} x \\ z \\ 0 \end{pmatrix} \teq \Zero \enspace ,
  \end{equation}
  where $\Id$ is a tropical identity matrix with $0$ on the diagonal and $\Zero$ elsewhere.

  By Lemma~\ref{lem:convexity+cone+preimage} and the tropical convexity of $\SetOf{x \in \TTpm^n}{\bigoplus_{j \in [n]} x_j = 0, x \geq \Zero}$, the set of $(x,z) \in \TTpm^{n+d}$ fulfilling~\eqref{eq:equation+system+inner+hull} is the intersection of tropically convex sets and, by Proposition~\ref{prop:immediate+properties+convex+sets}\eqref{item:intersection+convex+sets}, tropically convex as well. 

  Hence, we can use Theorem~\ref{thm:non+strict+signed+fourier+motzkin} to successively project out the $x$-variables.
  As the inequalities arising from a projection may contain balanced coefficients, we use Proposition~\ref{prop:resolving+balanced+coefficiens+teq} to replace them by signed coefficients.
  This yields a description of $\tconv(A)$ by non-strict
  inequalities. Here, we use that resolving a balanced coefficient in
  an inequality leads to a tighter constraint. 
%Thus, if the
 % description with balanced coefficients coincided with the coordinate
 % projection, and the constraints obtained from 
%% \todo[inline]{LV: added last sentence. I'm not sure if it helps or
%%   confuses more. I tried to add more detailed descriptions but those
%%   ended up worse...}

  \smallskip

  If the intersection of closed tropical halfspaces $M$ is tropically convex, then its intersection with any orthant is tropically convex.
  By the tropical Minkowski-Weyl theorem in the non-negative orthant~\cite{GaubertKatz:2011}, each of the parts in the orthants are finitely generated.
  Taking the tropical convex hull of the union of all these generators yields $M$, as $M$ is tropically convex. 
\end{proof}

%% \begin{claim}
%%   For each orthant $O$, let $\mathcal{H}_O$ be a set of halfspaces whose intersection is $\tconv(M) \cap O$. Let $U_O \subseteq \mathcal{H}_O$ be the set of halfspaces whose complement contains a point of $\tconv(M)$.
%%   Then
%%   \[
%%   \tconv(M) = \bigcap_{O} \bigcap_{H \in \mathcal{H}_O \setminus U_O} H \enspace.
%%   \]
%%   %% The union of the closed halfspaces arising from the exterior description in each orthant gives an exterior description of the whole signed hull. 
%% \end{claim}
%% \begin{proof}

%%   Should follow from Claim~\ref{claim:signed+Minkowski+Weyl}.
  %% By construction, $\tconv(M)$ is contained in the intersection.

  %% \smallskip

  %% Let $p$ be a point which is in the intersection but not in $\tconv(M)$ and let $O$ be an orthant containing $p$ (there could be multiple, if the point lies on the boundary). Then there is a halfspace $H \in U_O$ which contains $p$ in its complement. By the definition of $U_O$, there is a point $q \in \tconv(M)$ in another orthant $O'$ which is also contained in the complement of $H$.

  %% Let $s$ be the unique support-minimal point in the intersection of the boundary of $O$ with $\tconv(p,q)$. 
%% \end{proof}

\section{Connection with tropical convexity in $\TTmax$} \label{sec:generators+in+orthants}

The intersection of a signed tropically convex set with an orthant is just a tropically convex set over $\TTmax$.
This allows to study signed tropical convex sets through the existing theory of unsigned tropically convex sets with the hull~\eqref{eq:positive+convex+hull} in each orthant of $\TTpm^d$.
The proofs of the main statements are deferred to
Section~\ref{sec:orthant+proofs}; these are based on the duality between the non-negative kernels
and open tropical cones.

We fix a (finite) set $M \subseteq \TTpm^d$ and interpret $M$ as a matrix whose columns list the points.
We will augment $M$ by iteratively intersecting line segments between points in $M$ with the coordinate hyperplanes to get a matrix $\zeta(M)$.

To make the construction of $\zeta(M)$ formal, we start with a slight modification of~\eqref{eq:incidence+matrix+elimination}.
We fix $A = (a_{ij}) \in \TTpm^{d \times n}$ and a row index $i \in [d]$.

For two columns $u$ and $v$ of $A$ with $u_i > \Zero$ and $v_i < \Zero$, we set
\begin{equation} \label{eq:convex+cancellation}
  \begin{aligned}
\lambda^+(u,v) &= \ominus v_i^{\odot-1} \odot (u_i^{\odot-1} \oplus \ominus v_i^{\odot-1})^{\odot-1} \text{ and } \\
\lambda^-(u,v) &= u_i^{\odot-1} \odot (u_i^{\odot-1} \oplus \ominus v_i^{\odot-1})^{\odot-1} \enspace .
  \end{aligned}
\end{equation}
These scalars fulfill $\lambda^+(u,v) \oplus \lambda^-(u,v) = 0$ and $\lambda^+(u,v),\lambda^-(u,v) \geq \Zero$.

we define $T^{(i)} = (t_{j,p}) \in \{0,\Zero\}^{n \times ((J^+ \cup J^{\bullet}) \times (J^{\bullet} \cup J^-) \cup J^0)}$, using the decomposition from~\eqref{eq:index+sets+signs} for $A$, as the incidence matrix
  \begin{equation} \label{eq:modified+incidence+matrix+elimination}
  t_{j,p} =
  \begin{cases}
    \lambda^+(a^{(k)},a^{(\ell)}) & j = k  \text{ for } p = (k,\ell) \in (J^+ \cup J^{\bullet}) \times (J^{\bullet} \cup J^-)  \\
    \lambda^-(a^{(k)},a^{(\ell)}) & j = \ell \text{ for } p = (k,\ell) \in (J^+ \cup J^{\bullet}) \times (J^{\bullet} \cup J^-)  \\
    0 & j = p \text{ for } p \in J^0 \\
    \Zero & \text{ else }
  \end{cases}
  \enspace .
  \end{equation}

  \smallskip

    In Theorem~\ref{thm:balanced+fourier+motzkin}, we could also have used the matrix from~\eqref{eq:modified+incidence+matrix+elimination} instead of multiplying with both $S^{(i)}$ and $T^{(i)}$.
  This is technically slightly more complicated but has the advantage that the resulting columns of the product with $A$ are tropical convex combinations of the original columns.
  This motivates the following. %to define the matrix $\zeta_i(A)$ for an arbitrary row $i \in [d]$.
\begin{definition} \label{def:ith+elimination+matrix} 
  The matrix $\zeta_i(A)$ is the \emph{$i$th elimination matrix} of $A$ and it arises from $\xi(A \odot T)$ by replacing the $i$th row with $\Zero$. 
\end{definition}

This allows to describe generators for the intersection of a tropical convex hull with a coordinate hyperplane $H_i := \SetOf{x \in \TTpm^d}{x_i = \Zero}$ for $i \in [d]$.

\begin{proposition} \label{prop:intersection+coordinate+hyperplane}
  The intersection $\tconv(A) \cap H_i$ is generated by $\zeta_i(A)$.
\end{proposition}

Additionally, we get generators for the intersection with an orthant. 

\begin{proposition} \label{prop:intersection+orthant}
  The intersection $\tconv(A) \cap \Tgeq^d$ is generated by
  \[
  \left(A \cup \bigcup_{i \in [d]} (\tconv(A) \cap H_i) \right) \cap \Tgeq^d \enspace .
  \]
\end{proposition}

Now, we have to combine all possible ways of descending in a chain of intersections of coordinate hyperplanes. 
Each permutation $\sigma$ in the symmetric group on $d$ elements $S_d$ gives rise to a sequence of matrices $\zeta_{\sigma(1)}(M)$, $\zeta_{\sigma(2)}(\zeta_{\sigma(1)}(M))$ until $\zeta_{\sigma(d)}(\dots(\zeta_{\sigma(2)}(\zeta_{\sigma(1)}(M))\dots)$.
We denote the concatenation of these $d$ matrices by $\zeta_{\sigma}(M)$.
The concatenation of the matrices for all $\sigma \in S_d$ forms the matrix $\zeta(M)$.

\begin{theorem} \label{thm:generators+for+all+orthants}
  The convex hull $\tconv(M)$ of $M$ is the union
  \begin{equation}
    \bigcup_{O\text{ closed orthant of } \TTpm^{d}} \tconv(O \cap \zeta(M)) \enspace .
  \end{equation}
\end{theorem}
\begin{proof}
  By Proposition~\ref{prop:intersection+orthant}, we know that $\tconv(M)$ is generated by the projections on the boundary of the orthants and the generators in the interior.
  Iteratively applying Proposition~\ref{prop:intersection+coordinate+hyperplane} to the intersection of coordinate hyperplanes yields the claim.
\end{proof}

\begin{example}
  Looking at the points from Example~\ref{ex:first+convex+hull+critical+points}, we see how we can determine the tropical convex hull of $\{(3,3),(\ominus 1, \ominus 0),(\ominus 4, \ominus 2)\}$.
  It is the union of the tropical convex hulls
  \begin{align*}
    &\tconv\left(\{(3,3),(\Zero,3),(\Zero,1)\}\right), \quad \tconv\left(\{(\ominus 1, 0),(\Zero, 1),(\Zero, 3),(\ominus 4,\Zero)\}\right) \\
    &\tconv\left(\{(\ominus 1, \ominus 0),(\ominus 4, \ominus 2),(\ominus 1,\Zero),(\ominus 4,\Zero)\}\right) \enspace .
  \end{align*}
  
  \smallskip

  On the other hand, to get the tropical convex hull $\tconv((0,0),(\ominus -2, \ominus -2))$ in Figure~\ref{subfig:tropical+exploded+line+segment}, one needs actual multi-valued cancellation.
  
  We get
  \[
  \zeta_1\left(\pvec{0}{0},\pvec{\ominus -2}{\ominus -2}\right) = \left\{\pvec{\Zero}{-2},\pvec{\Zero}{\ominus -2}\right\}\text{ and }
  \zeta_2\left(\pvec{0}{0},\pvec{\ominus -2}{\ominus -2}\right) = \left\{\pvec{-2}{\Zero},\pvec{\ominus -2}{\Zero}\right\}.
  \]
  From applying $\zeta_1(\zeta_2(.))$, we additionally obtain $\{(\Zero,\Zero)\}$. 
  
 For the positive orthant, this yields the generators $\{(\Zero,\Zero),(\Zero,-2),(-2,\Zero),(0,0)\}$.
  The other orthants are derived analogously. 
\end{example}

\begin{corollary}
  Tropically convex sets are contractible.
\end{corollary}
\begin{proof}
  The space $\TTpm^d$ inherits the topology of $\RR^d$ via the map $\slog \colon x \mapsto \sgn(x) \log(|x|)$, where the origin is mapped to the all-$\Zero$-point.
  As tropically convex sets in all orthants are contractible~\cite[Theorem 2]{DevelinSturmfels:2004}, we can contract to the boundary of the orthants. The claim follows by induction on the dimension $d$. 
  %Fixing an arbitrary point $x_0$ in a tropically convex set $M \subseteq \TTpm^d$ yields a homotopy along the sets 
\end{proof}

For the definition of the covector decomposition in $\left(\RR \cup \{-\infty\}\right)^d$ and its connection with regular subdivisions of $\Delta_d \times \Delta_n$ we refer the reader to~\cite{JoswigLoho:2016}.

\begin{corollary}[Covector decomposition]
  The combinatorics of the tropically convex hull of a matrix $A \in \TTpm^{d \times n}$ can be described by $2^d$ regular subdivisions of $\Delta_d \times \Delta_n$. 
\end{corollary}

\subsection{Proofs}\label{sec:orthant+proofs}

\begin{proof}[Proof of Proposition~\ref{prop:intersection+coordinate+hyperplane}]
%  For simplification of notation, we can assume that $i = d$.
  Using the construction in~\eqref{eq:modified+incidence+matrix+elimination} and the fact that $\ominus|z|, |z| \in \Uncomp(z)$ for $z \in \Tzero$, the inclusion $\zeta_i(A) \subseteq \tconv(A) \cap H_i$ follows from Definition~\ref{def:inner+hull}.
  %With Proposition~\ref{prop:generation+by+line+segments}, this
  Example~\ref{ex:coordinate+hyperplane} and Proposition~\ref{prop:immediate+properties+convex+sets}.\ref{item:intersection+convex+sets} imply that $\tconv(\zeta_i(A)) \subseteq \tconv(A) \cap H_i$.
  
  \smallskip

  For the other inclusion, assume that there is a point $z \in \tconv(A)$ with $z_i = \Zero$ which is not contained in $\tconv(\zeta_i(A))$. By Proposition~\ref{prop:equivalence+hull+solution}, this implies that
  \[
\nnker
  \begin{pmatrix}
    0 & \ominus 0 \\
    \zeta_i(A) & \ominus z 
  \end{pmatrix}
  = \emptyset \enspace . 
  \]
  The Farkas Lemma~\ref{thm:farkas+lemma} implies that
\[
\sep \begin{pmatrix}
    0 & \ominus 0 \\
    \zeta_i(A) & \ominus z 
\end{pmatrix} \neq  \emptyset  \enspace .
\]
Furthermore, we get
\[
\sep( \zeta_i \begin{pmatrix}
  0 & \ominus 0 \\
  A & \ominus z 
\end{pmatrix}) \neq \emptyset
 \]
as, because of $z_i = \Zero$, the last column is unchanged by $\zeta_i$ and the first row remains the same due to the definition of $\lambda^+(u,v)$ and $\lambda^-(u,v)$ for~\eqref{eq:convex+cancellation}.

  However, by Corollary~\ref{cor:signed+fourier+motzkin}, then also
  \[
     \sep \begin{pmatrix}
    0 & \ominus 0 \\
    A & \ominus z 
      \end{pmatrix}
    \neq   \emptyset \enspace .
    \]
    Using again the Farkas Lemma~\ref{thm:farkas+lemma}, this implies $z \not\in \tconv(A)$, a contradiction. 
\end{proof}

To derive the next claim, we start with a more precise description of a signed tropical line segment. 
There is a natural bijection between $\Delta_2 = \SetOf{(\nu,\mu) \in \Tgeq^2}{\max(\nu,\mu) = 0}$ and $\Tclosed = \RR \cup \{-\infty,\infty\}$ given by
\[
(\nu,\mu) \mapsto \mu - \nu \enspace .
\]
We denote the inverse image of an element $\eta \in \Tclosed$ with respect to this map by $\Psi(\eta)$.
This leads to a parametrization of a tropical line segment for $a,b \in \TTpm^d$ via $\tconv(a,b) = \SetOf{L_{\eta}(a,b)}{\eta \in \Tclosed}$ where $L_{\eta}(a,b) = \Psi(\eta)_0 \odot a \oplus \Psi(\eta)_1 \odot b$. Note that $L_{-\infty}(a,b) = a$ and $L_{\infty}(a,b) = b$.

\begin{proof}[Proof of Proposition~\ref{prop:intersection+orthant}]
  Let $z \in \tconv(A) \cap \Tgeq^d$ be an element of $\Uncomp(A \odot \lambda)$ with $\lambda \in \Delta_d$.
  We consider the tropical line segments from $z$ to the columns of $A$.
  For a fixed column $a^{(j)}$ of $A$ with $a^{(j)} \not\in \Tgeq^d$, there is a minimal $\eta \in \Tclosed$ such that a component of $L_{\eta}(z,a^{(j)})$ is balanced.

  \smallskip
  
  {\it Intermediate claim I:} All entries of $L_{\eta}(z,a^{(j)})$ are either balanced or in $\Tgeq$.
  For an arbitrary row $i \in [d]$, the expression $\Psi(\eta)_1 \odot z_i \oplus \Psi(\eta)_2 \odot a^{(j)}_i$ is in $\Tgeq$ for $\eta = -\infty$. 
  The claim now follows from the piecewise definition of the addition in terms of the absolute values.

  \smallskip

  Using Proposition~\ref{prop:generation+by+line+segments}, we see that the point $b^{(j)}$ obtained from $L_{\eta}(z,a^{(j)})$ by replacing all balanced entries with $\Zero$ is in $\tconv(A) \cap \Tgeq^d$. For $a^{(j)} \in \Tgeq^d$ we set $b^{(j)} = a^{(j)}$.

  \smallskip

  {\it Intermediate claim II:} The point $z$ is in the convex hull of $\SetOf{b^{(j)}}{j \in [n]}$. It is enough to show that
  \[
  \nnker(
  \begin{pmatrix}
    0 & \ldots  & 0 & 0 & \ominus 0 \\
    a^{(1)} & \ldots & a^{(n-1)} & b^{(n)} & \ominus z 
    \end{pmatrix} ) \neq \emptyset
  \]
  because than we can iteratively replace the columns $a^{(i)}$ by $b^{(i)}$.
  Let $b^{(1)} \in  \nu \odot z \oplus \mu \odot a^{(1)}$. 
  Pick an element $x$ scaled such that $x_1 = \mu$.
%  Note furthermore that $(\Zero,\ldots,\Zero,\nu,\nu)$ is an element of 
%  \[
%  \nnker(
%  \begin{pmatrix}
%    0 & \ldots  & 0 & 0 & \ominus 0 \\
%    a^{(1)} & \ldots & a^{(n-1)} & z & \ominus z 
%    \end{pmatrix} ) \enspace .
%  \]
  Then
  \begin{equation*}
  \begin{pmatrix}
    \nu & \Zero & \ldots & \Zero & \ominus \nu \\
    \nu \odot z & \Zero & \ldots & \Zero & \nu \odot \ominus z 
  \end{pmatrix} \oplus
  \begin{pmatrix}
    0 & 0 & \ldots  & 0 & \ominus 0 \\
    a^{(1)} & a^{(2)} & \ldots & a^{(n)} & \ominus z 
  \end{pmatrix} \odot x 
  \end{equation*}
  is in $\Tzero^d$.
  Therefore $(0,\ldots,0)$ is in the non-negative kernel of 
  \begin{equation} \label{eq:sum+canceling+combination}
    \begin{aligned}
  \begin{pmatrix}
    \nu \oplus x_1 & x_2 & \ldots  & x_n & \ominus (\nu \oplus x_{n+1})\\
    \nu \odot z \oplus x_1 \odot a^{(1)} & x_2 \odot a^{(2)} & \ldots & x_n \odot a^{(n)} & (\nu \oplus x_{n+1}) \odot (\ominus z)
  \end{pmatrix} = \\
    \begin{pmatrix}
    0 & x_2 & \ldots  & x_n & \ominus (\nu \oplus x_{n+1})\\
    \nu \odot z \oplus \mu \odot a^{(1)} & x_2 \odot a^{(2)} & \ldots & x_n \odot a^{(n)} & (\nu \oplus x_{n+1}) \odot (\ominus z)
  \end{pmatrix}
    \end{aligned}
  \end{equation}
  For fixed $i \in [d]$, if $\nu \odot z_i \oplus \mu \odot a^{(1)}_i$ is not balanced or the maximum absolute value is attained somewhere else in the row, we can replace it by $b^{(1)}_i$ and $(0,\ldots,0)$ is still in the non-negative kernel.
  
  Otherwise, $\ominus \nu \odot z_i = \mu \odot a^{(1)}_i$ and $\nu \oplus x_{n+1} = \nu$.
  But $(0,\ldots,0)$ is also in the non-negative kernel of 
  \[
  \begin{pmatrix}
     \mu & x_2 & \ldots  & x_n & \ominus x_{n+1}\\
    \mu \odot a^{(1)} & x_2 \odot a^{(2)} & \ldots & x_n \odot a^{(n)} & x_{n+1} \odot (\ominus z)
  \end{pmatrix} \enspace .
  \]
  Since $\mu \odot a^{(1)}_i = \ominus \nu \odot z_i$ has the same sign as $\ominus x_{n+1} \odot z_i$ and we have $\nu \geq x_{n+1}$, there has to be an $\ell \in [n]$ such that $x_{\ell} \odot a^{(\ell)}_i = \ominus \mu \odot a^{(1)}_i = \nu \odot z_i$. Therefore, we can replace $\nu \odot z_i \oplus x_1 \odot a^{(1)}$ by $\Zero$ and $(0,\ldots,0)$ remains in the non-negative kernel.

  Therefore, $(0,x_2, \ldots, x_n, (\nu \oplus x_{n+1}))$ is in the non-negative kernel of
  \[
      \begin{pmatrix}
    0 & 0 & \ldots  & 0 & \ominus 0\\
    b^{(1)} & a^{(2)} & \ldots & a^{(n)} & \ominus z
  \end{pmatrix} \enspace .
      \]
      This finishes the proof of claim II. 
  
  %% \smallskip
  
  %% The statement now follows from Proposition~\ref{prop:intersection+coordinate+hyperplane}.
\end{proof}

\section{Acknowledgements}

We thank Daniel Dadush for insisting on balanced numbers and for inspiring
discussions on the tropical Fourier-Motzkin elimination. 
We thank Xavier Allamigeon as well as St{\'e}phane Gaubert for communicating their related work.
We are grateful to Matthias Schymura for helping to get an intuition for the unintuitive line segments.

\iffull\else
\appendix

\section{Missing proofs}

\bpteq*
\proofbpteq

\beqcomb*
\proofbeqcomb

\genlineseg*
\proofgenlineseg

\conhc*
\proofconhc

\eqhullsol*
\proofeqhullsol

\lifthull*
\prooflifthull

\elimstrict*
\proofelimstrict

\elimnonstrict*
\proofelimnonstrict
\section{Other notions of tropical convexity}\label{app:b-conv}

Parallel to the development of tropical convexity, the more general notion of $\BB$-convexity was developed starting with~\cite{BriecHorvath:2004}.
The notion of $\BB$-convexity boils down to convexity defined over the semiring $\RR_{\geq 0}$ with operations `$\oplus$' $=$ `$\max$' and `$\odot$' $=$ `$\cdot$', see~\cite[Theorem 2.1.1]{BriecHorvath:2004}.
Taking logarithms transforms these operations to `$\oplus$' $=$ `$\max$' and `$\odot$' $=$ `$+$' on $\RR \cup \{-\infty\}$.
This gives rise to a transferred version of $\BB$-convexity on $\TTpm$ by considering the images of $\BB$-convex sets in $\RR^d$ under the map $\slog \colon x \mapsto \sgn(x) \log(|x|)$.

 The following example shows that our notion of signed tropical convexity is an even more restrictive notion than $\BB$-convexity and $\BB^{\sharp}$-convexity~\cite{Briec:2015}.
\begin{example} \label{ex:different+convexities}
  The tropical convex hull of $A = \{(\ominus 2, \ominus 1), (2,1)\}$ is the set
  \[
  [\ominus 2, 2] \times [\ominus 1,1] \enspace . %\cup 
  \]
  However, the set $\Co^r(A)$ is
  \[
  L = \SetOf{(2\odot\lambda,\lambda)}{\lambda \in [\ominus 1,1]} \enspace .
  \]
  for all $r \in \NN$.
  In particular, also $\Co^{\infty}(A)$ equals $L$.
  This implies that $\BB(L) = L$.
  We depict both in Figure~\ref{fig:BB+convex+trop+convex}.
   Hence, $\tconv(A)$ strictly contains $\BB(A)$. 
   Furthermore,~\cite[Corollary 4.2.4]{Briec:2015} shows that $L$ is also $\BB^{\sharp}$-convex.

   Interestingly, the set $L$ is also the image under the signed valuation of the set
  \[
  \conv\left( \pvec{-t^2}{-t}, \pvec{t^2}{t} \right) \enspace .
  \]
  Here, we mean the convex hull over the Puiseux series $\pseries{R}{t}$.
  So $L$ is the tropicalization of a single line segment while our hull construction yields the union of line segments whose spanning points tropicalize to $A$, as we saw in Section~\ref{subsec:puiseux+lifts}.
  For example, we get the set $\{\ominus 2\} \times [\ominus 1, 1] \cup [\ominus 2, 2] \times \{1\}$ as the tropicalization of $\conv\left(\pvec{-2t^2}{-t},\pvec{t^2}{2t}\right)$.
\end{example}

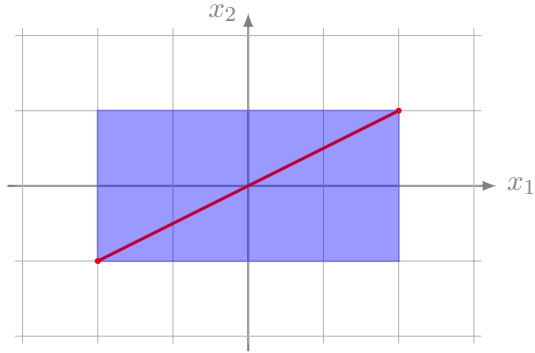
\begin{figure}[ht]
  \begin{tikzpicture}

  \Koordinatenkreuz{-3.2}{3.3}{-2.2}{2.3}{$x_1$}{$x_2$}
  \Gitter{-3.1}{3.1}{-2.1}{2.1}

%  \coordinate (center) at (1.5,1.5){};
  \coordinate (pp) at (2,1){};
  \coordinate (mp) at (-2,1){};
  \coordinate (pm) at (2,-1){};
  \coordinate (mm) at (-2,-1){};
  
%  \draw[Linesegment] (center) -- (pp);
  \draw[fill,Linesegment,opacity=0.4] (pp) -- (mp) -- (mm) -- (pm) -- cycle;

  \node[Endpoint] at (pp){};
  \node[Endpoint] at (mm){};

  \draw[purple, very thick] (pp) -- (mm);
  
\end{tikzpicture}
  \caption{Distinction between $\BB$-convex line and tropical line segment through the origin. }
  \label{fig:BB+convex+trop+convex}
\end{figure}

 \begin{remark} \label{rem:cancellative+sum}
It is tempting to define a \emph{cancellative sum} for two numbers $a, b \in \TTpm$ by
\[
a \overline{\oplus} b =
\begin{cases}
  a & |a| > |b| \\
  b & |b| > |a| \\
  a & a = b \\
  \Zero & a = \ominus b
\end{cases}
\enspace .
\]
This can be extended componentwise to $\TTpm^d$.

An iterative version of this construction is used in~\cite{Briec:2015}.
A conceptional drawback of the cancellative sum is that it is not associative, as the example
  \[
  0 \cancsum (\ominus 0 \cancsum -1) = 0 \cancsum \ominus 0 = \Zero \neq -1 = \Zero \cancsum -1 = ( 0 \cancsum \ominus 0) \cancsum -1 % \enspace ,
  \]
  shows.
%  We use a similar but multi-valued version in Section~\ref{subsec:fourier-motzkin} for~\eqref{eq:multi-valued+cancellation}.
\end{remark}

   \begin{remark}
     Recall that Theorem~\ref{thm:generators+for+all+orthants} gave a way to determine the whole tropical convex hull from the tropical convex hull in each orthant.
     For sufficiently generic matrices, one can use the cancellative sum from Remark~\ref{rem:cancellative+sum}.
  If there are no antipodal points, no balanced coefficients arise in the elimination and the map $\xi$ used for Theorem~\ref{thm:replace+balanced} is just the identity. 
  Hence, the iterative construction of a single intersection point with a coordinate hyperplane suffices.      
   \end{remark}

\fi

\bibliographystyle{amsplain}
\bibliography{main}

\end{document}